\documentclass[11pt]{amsart}
\usepackage
[    left=3.5cm,
    right=3.5cm,
]
{geometry}

\usepackage{tikz}
\usetikzlibrary{matrix}

\usepackage{lineno}
\usepackage{amssymb}
\usepackage{verbatim}
\usepackage{array}
\usepackage{latexsym}
\usepackage{enumerate}
\usepackage{amsmath}
\usepackage{amsfonts}
\usepackage{amsthm}
\usepackage{color}
\usepackage[english]{babel}

\makeatletter
\usepackage{comment}
\let\wfs@comment@comment\comment
\let\comment\@undefined

\usepackage{changes}
\let\wfs@changes@comment\comment
\let\comment\@undefined

\newcommand\comment{%
    \ifthenelse{\equal{\@currenvir}{comment}}
    {\wfs@comment@comment}
    {\wfs@changes@comment}%
}

\definechangesauthor[name=Korchmaros, color=green]{Korchmaros}
\definechangesauthor[name=Pietro, color=red]{Pietro}
\definechangesauthor[name=Herivelto, color=blue]{Herivelto}

\newtheorem{theorem}{Theorem}[section]
\newtheorem{proposition}[theorem]{Proposition}

\newtheorem{lemma}[theorem]{Lemma}
\newtheorem{lem}[theorem]{Lemma}

{\theoremstyle{definition}
}

\newtheorem{rem}[theorem]{Remark}

\def\cC{\mathcal C}
\def\cD{\mathcal D}
\def\cE{\mathcal E}
\def\cF{\mathcal F}
\def\cG{\mathcal G}

\def\cH{\mathcal H}
\def\cK{\mathcal K}
\def\cL{\mathcal L}

\def\cU{\mathcal U}

\def\cX{\mathcal X}
\def\cY{\mathcal Y}

\def\PG{{\rm{PG}}}
\def\GL{{\rm{GL}}}

\def\deg{\mbox{\rm deg}}
\def\det{\mbox{\rm det}}

\def\fq{{\mathbb F}_q}

\def\fqq{{\mathbb F}_{q^2}}

\newcommand{\PSL}{\mbox{\rm PSL}}

\newcommand{\PGL}{\mbox{\rm PGL}}
\newcommand{\AGL}{\mbox{\rm AGL}}

\newcommand{\PSU}{\mbox{\rm PSU}}

\newcommand{\PGU}{\mbox{\rm PGU}}

\newcommand{\AG}{\mbox{\rm AG}}

\newcommand{\aut}{\mbox{\rm Aut}}

\newcommand{\ha}{{\textstyle\frac{1}{2}}}

\begin{document}
\title{Plane curves with a large linear automorphism group in characteristic $p$}
\thanks{This research was performed within the activities of  GNSAGA - Gruppo Nazionale per le Strutture Algebriche, Geometriche e le loro Applicazioni of Italian INdAM.
 The third author  was partially supported by FAPESP-Brazil, grant 2017/18776-6.}
 \thanks{Herivelto Borges is with the Instituto de Ci\^encias Matem\'aticas e de Computa\c{c}\~ao, Universidade de S\~ao Paulo, S\~ao Carlos, SP 13560-970, Brazil.\\ e-mail: hborges@icmc.usp.br}
 \thanks{ G\'abor Korchm\'aros is with the Dipartimento di Matematica, Informatica ed Economia, Universit\`a  degli Studi della Basilicata,   Potenza (PZ)  85100, Italy. \\ e-mail: gabor.korchmaros@unibas.it}
 \thanks{Pietro Speziali is with the Instituto de Matem\'atica, Estat\'istica e Computa\c{c}\~ao Cient\'ifica, Universidade Estadual de Campinas, Campinas, SP 13083-859, Brazil.\\ e-mail: speziali@unicamp.br}



 \author{Herivelto Borges}

 \author{G\'abor Korchm\'aros}

 \author{Pietro Speziali}

\begin{abstract}  Let $G$ be a subgroup of the three dimensional projective group $\PGL(3,q)$ defined over a finite field $\mathbb{F}_q$ of order $q$, viewed as a subgroup of $\PGL(3,K)$ where $K$ is an algebraic closure of $\mathbb{F}_q$. For the seven nonsporadic, maximal subgroups $G$ of $\PGL(3,q)$, we investigate the (projective, irreducible) plane curves defined over $K$ that are left invariant by $G$. For each, we compute the minimum degree $d(G)$ of $G$-invariant curves, provide a classification of all $G$-invariant curves of degree $d(G)$, and determine the first gap $\varepsilon(G)$ in the spectrum of the degrees of all $G$-invariant curves. We show that the curves of degree $d(G)$ belong to a pencil depending on $G$, unless they are uniquely determined by $G$. We also point out that $G$-invariant curves of degree $d(G)$  have particular geometric features such as Frobenius nonclassicality and an unusual variation of the number of $\mathbb{F}_{q^i}$-rational points. For most examples of plane curves left invariant by a large subgroup of $\PGL(3,q)$, the whole automorphism group of the curve is linear, i.e., a subgroup of $\PGL(3,K)$. Although this appears to be a general behavior, we show that the opposite case can also occur for some irreducible plane curves, that is,  the curve has a  large  group of  linear automorphisms, but  its  full automorphism group  is  nonlinear.
\end{abstract}
\maketitle
\section{Introduction}
Let $K$ be an algebraically closed field. A classical result on automorphisms of (projective, irreducible) plane curves defined over $K$ is that,  if the curve is  nonsingular and of degree at least four, then its automorphism group is linear, that is, a subgroup of the three-dimensional projective general linear group $\PGL(3,K)$. By a result of B. Segre, this theorem holds independently of  the characteristic of the field $K$,  see\cite[Theorem 11.29]{hirschfeld-korchmaros-torres2008}. The hypothesis of nonsingularity cannot be dropped, nevertheless, singular curves may have many linear automorphisms. If this occurs,  the curve often happens to have useful applications to finite geometry and coding theory.

In this paper, we work over an algebraic closure $K$ of a prime field of characteristic $p>0$. For a given finite subgroup $G$ of $\PGL(3,K)$, let $d(G)$ denote the smallest integer that is the degree of a $G$-invariant irreducible plane curve other than a line. Clearly, $d(G)$
does not change if we replace $G$ with another group from  the conjugacy class of $G$ in $\PGL(3,K)$. We raise  the following three questions for subgroups of $\PGL(3,K)$.
\begin{enumerate}
\item[(i)] find $d(G)$ for groups $G$ of larger order compared to $d$;
\item[(ii)] find a positive integer $\varepsilon(G)$ depending on $q$ such that no $G$-invariant irreducible plane curve of degree $d(G)+\varepsilon(G)$ exists.
\item[(iii)] determine all $G$-invariant irreducible plane curves of degree  $d(G)$.
\end{enumerate}
Note  that the  finiteness of $G$ implies $G\le \PGL(3,\fq)=\PGL(3,q)$ for some power $q$ of $p$. We choose the smallest such $q$ and address the above three questions for maximal, nonsporadic, subgroups of $\PGL(3,q)$ for a given $q$. We provide a complete solution for each. A summary of the results is provided in Section \ref{rr}, where some consequences regarding  the geometry of low degree $G$-invariant irreducible plane curves are also stated, especially from the point of view of the St\"ohr-Voloch theory.

The essential idea underlying our investigation has  its origin in a basic result from classical invariant theory, which  finds that if $\Gamma\le {\rm GL}(3,K)$ is a linear group and
$F_1,F_2 \in K[X,Y,Z]$ are both $\Gamma$-invariant forms (homogeneous polynomials) of the same degree $d$,   then any non-trivial linear combination $F=\lambda_1F_1+\lambda_2 F_2$ is also a $\Gamma$-invariant form. The projective version concerns the projective linear group $G=\Gamma/Z(\Gamma)$,  together with  the pencil $\Lambda$ generated by the curves $\cC_1,\cC_2$ of homogeneous equations $F_1=0,F_2=0$, respectively.  It states that $\Lambda$ is $G$-fixed, that is, each curve $\cC$ of degree $d$ of equation $F=\lambda_1F_1+\lambda_2F_2=0$ is  $G$-invariant. A useful refinement that still ensures that the pencil is  $G$-fixed  requires  only that the rational function $F_1/F_2$  is $\Gamma$-invariant. Once such a $G$-fixed pencil $\Lambda$ has been found, if $\cD$ is any $G$-invariant plane curve of degree $d\ge d(G)$,  then $|G|$ is a lower bound on the number of common points of $\cD$ with a generically chosen curve $\cC$ from $\Lambda$. Comparison of this lower bound with the upper bound derived  from the B\'ezout theorem yields $|G|\le d(G)d$. In other words, if $d<d(G)/|G|$,  then $\cD$ has to be a curve in the pencil $\Lambda$. Thus the problem arises whether a $G$-fixed pencil of curves of degree $d(G)$ exist for each of the maximal, nonsporadic subgroups of $\PGL(3,q)$, namely $G\cong\PGL(3,q),\AGL(2,q), \overline{\AGL}(2,q),\PSU(3,d), NS_{q}, \Delta_q$, and $\PGL(2,q)$  for $q$ odd. The definitions of these groups are found  in Section \ref{sec:pgl}.
 We will show that the answer to this  is affirmative. In particular, the exact value of $d(G)$ is computed, which resolves Question (i) for these groups.
Furthermore, $\varepsilon(G)\ge d(G)(1-1/|G|)$. Actually, as we will show, equality attains and hence a complete solution for Question(ii) is obtained for each of the above seven groups. The study of Question (iii) requires a closer look at  the special features of the $G$-fixed pencils. They  differ depending on which of the above seven groups is taken for $G$. Nevertheless,  a case-by-case analysis is enough to obtain a complete solution also for this question.


\section{Background and preliminary results}
\label{sbp}
\subsection{Plane curves and their automorphisms}
For definitions and basic properties of plane algebraic curves, see \cite[Chapters 1-5]{hirschfeld-korchmaros-torres2008}.
In the projective plane $\PG(2,q)$ equipped with a projective reference system in coordinates $(X:Y:Z)$, a plane (algebraic) curve $\cC$ of homogeneous equation $F(X,Y,Z)=0$, with a homogeneous polynomial (or form) $F(X,Y,Z)\in \fq[X,Y,Z]$, consists of all points $P=(a:b:c)$ such that $F(a,b,c)=0$. As  customary, we regard $\cC$ as a curve over a fixed algebraic closure $K$ of $\fq$. Doing so, the points $P=(a:b:c)$ of $\PG(2,K)$ for which $F(a:b:c)=0$ are exactly the points of $\cC$. An affine equation of $\cC$ is $f(X,Y)=0$ where $f(X,Y)=F(X,Y,1)$. We use standard notation. The \emph{degree} of $\cC$  is $\deg(F(X,Y,Z))$. A \emph{component} of $\cC$ is a curve $\cG$ of homogeneous equation $G(X,Y,Z)=0$  such that $G(X,Y,Z)$ divides $F(X,Y,Z)$. A curve $\cC$ is (absolutely) \emph{irreducible} if $F(X,Y,Z)$ is irreducible over $K$; otherwise, $\cC$ is \emph{reducible} as it splits in irreducible curves, the \emph{components} of $\cC$.

Given two plane curves $\cC$ and $\cF$, for any point $P\in \PG(2,K)$, the \emph{intersection multiplicity} $I(P,\cC\cap \cF)$ between $\cC$ and $\cF$ at $P$ is a main tool in the study of plane curves,  see \cite[Chapter 3]{hirschfeld-korchmaros-torres2008}. A point $P$ is a common point of $\cC$ and $\cF$ if and only if $I(P,\cC\cap \cF)\ge 1$, and $P$ is a \emph{singular point} of $\cC$, with multiplicity $r$, if $I(P;\cC\cap\ell)=r$ for all but finitely many lines $\ell$ through $P$. The exceptions are the \emph{tangent lines} to $\cC$ at $P$. They satisfy $I(P,\cC\cap \ell)>r$, and their number is positive but does not exceed $r$. By the fundamental B\'ezout theorem, for any two plane curves $\cC$ and $\cF$ without a common component, $\deg(\cC)\deg(\cF)= \sum I(P,\cF\cap \cG)$ where $P$ ranges over the set of all common points of $\cC$ and $\cF$. A \emph{pencil} generated by the plane curves $\cF$ and $\cG$ of homogeneous equations $F(X,Y,Z)=0$ and $G(X,Y,Z)=0$, respectively, consists of all plane curves $\cC_\lambda$  of homogeneous equations $\lambda F(X,Y,Z)+G(X,Y,Z)=0$ with $\lambda$ ranging over $K\cup\{\infty\}$ where $\cC_\infty=\cF$. For every point $P\in \PG(2,K)$ which is not a base point, i.e. it is not  a common point of any two curves in the pencil, there exists  a unique curve from the pencil passing through $P$.  A \emph{net} $\Gamma$ generated by such a pencil $\Lambda$ together with  further plane curves $\cH$, off the pencil with homogeneous equations $H(X,Y,Z)=0$, consists of the curves in $\Lambda$ together with the plane curves $\cC_\lambda$  of homogeneous equations $\lambda F(X,Y,Z)+\mu G(X,Y,Z)+ H(X,Y,Z)=0$ with $(\lambda,\mu)$ ranging over $K\times K$. In other words, the net $\Gamma$  consists of all plane curves of homogeneous equation $\lambda F(X,Y,Z)+\mu G(X,Y,Z)+ \tau H(X,Y,Z)=0$ with $(\lambda;\mu;\tau)$ ranging over the non-trivial homogeneous triples over $K$, i.e., over the points of $\PG(2,K)$.
For every incident point-line pair $(P,t)$ of $\PG(2,K)$, the net $\Gamma$ contains some curve $\cC$ through $P$ and tangent to $\cC$ at $P$.

A \emph{linear} automorphism of $\cC$ is a projectivity $\alpha$ represented by a $3\times 3$ regular matrix $A=(a_{ij})$ such that $\cC$ is $\alpha$-invariant, i.e., $F(X,Y,Z)=cF(\bar{X},\bar{Y},\bar{Z})$ for some $c\in K$ where the column vector $[X,Y,Z]$ is the product of $A$ by the column vector $[\bar{X},\bar{Y},\bar{Z}]$. The linear automorphisms of $\cC$ form the \emph{linear automorphism group} of $\cC$ which is a subgroup of $\PGL(3,K)$. If $\cC$ is an irreducible plane curve other than a line, then its linear automorphism group is finite. It should be noted that $\cC$ may have nonlinear automorphisms, i.e.,  there may exist a rational map $\alpha:(X,Y,Z)\mapsto (U,V,W)$ with $U,V,W\in K[X,Y,Z]$ homogeneous polynomials of the same degree $\ge 2$ such that $F(X,Y,Z)=SF(U,V,W)$ with some homogeneous polynomial $S\in K[X,Y,Z]$. The \emph{automorphism group} of $\cC$ consists of its linear and nonlinear automorphisms, and it is finite unless $\cC$ is either rational or elliptic.

We point out that irreducible plane curves with many linear automorphisms  often have a truly geometric property related to inseparable function field extensions that curves in zero characteristic cannot possess.
An irreducible plane curve $\cC$ is \emph{nonclassical} if it has infinitely many inflection points, that is, if $I(P,\cC\cap t)>2$ for infinitely many nonsingular points $P\in \cC$ where $t$ is the tangent line to $\cC$ at $P$. This impossible phenomena in zero characteristic was the starting point of the St\"ohr-Voloch theory in 1986 \cite{SV}; for recent results see \cite[Sections 8.5, 8.6]{hirschfeld-korchmaros-torres2008} and \cite{anazar,bor}.
Such nonclassicality occurs for $\cC$ if and only if there is a power $p^h\ge 3$ of $p$, denoted by $\varepsilon_2(\cC)$, such that $I(P,\cC\cap t)=p^h$ for infinitely many nonsingular points $P\in \cC$  with tangent line $t$ to $\cC$  at $P$. Let $\cC$ be an irreducible plane curve with homogeneous equation $F(X,Y,Z)=0$. Then $\cC$ is nonclassical if and only if there exist four homogeneous polynomials
$U_i(X,Y,Z),H(X,Y,Z)\in K[X,Y,Z]$ with the same degree such that
\begin{equation}
\label{eqnoc}
U_1(X,Y,Z)^{p^h}X+U_2(X,Y,Z)^{p^h}Y+U_3(X,Y,Z)^{p^h}Z=H(X,Y,Z)F(X,Y,Z)
\end{equation}
where $\varepsilon_2(\cC)=p^h$. If $\cC$ is nonsingular,  then $H(X,Y,Z)$ is constant. An irreducible plane curve $\cC$ defined over a finite field $\mathbb{F}_{p^k}$ is \emph{Frobenius nonclassical} if the tangent line $t$ to $\cC$ at $P$ passes through the Frobenius image of $P$ for infinitely many nonsingular points $P\in \cC$. Here, the \emph{Frobenius image} of the point $P(a:b:c)$ of $\PG(2,K)$ is the point $\Phi(P)=(a^{p^k}:b^{p^k}:c^{p^k})$ of $\PG(2,K)$. For odd $p$, any  Frobenius nonclassical curve is nonclassical. The converse does not hold in general, as $\cC$ is Frobenius nonclassical if and only if (\ref{eqnoc}) holds and there exists a homogeneous polynomial $L(X,Y,Z)\in K[X,Y,Z]$ such that
\begin{equation}
\label{eqnoc1}
U_1(X,Y,Z)X^{p^{k-h}}+U_2(X,Y,Z)Y^{p^{k-h}}+U_3(X,Y,Z)Z^{p^{k-h}}=L(X,Y,Z)F(X,Y,Z).
\end{equation}
Frobenius nonclassical curves are relevant in the study of plane curves of $\PG(2,p^k)$ which have many points in $\PG(2,p^k)$ compared with the degree of the curve.
In fact for $p>2$, if the number $N_{p^k}$ of points of an irreducible plane curve defined over $\mathbb{F}_{p^k}$ of degree $d$ exceeds $\ha d(p^k+d-1)$, then the curve is Frobenius nonclassical.
On the other hand, for an irreducible plane curve $\cC$ defined over $\mathbb{F}_{p^k}$ containing some points in $\PG(2,p^k)$, a lower bound on $N_{p^k}$ is $|G|/|G_P|$ where $G$ is a linear automorphism group of $\cC$ and $G_P$ is the stabilizer of $P$ in $G$. This shows that if $|G|$ is large (with $|G_P|$ relatively small) then $\cC$ has to be Frobenius classical.

\subsection{Maximal Subgroups of $\PGL(3,q)$}\label{sec:pgl}
The projective general linear group $\PGL(3,q)$ has order $(q^3-1)q^3(q^2-1)=(q^2+q+1)q^3(q+1)(q-1)^2$ and its proper subgroups have been  known for a long time, see  H.H Mitchell  \cite{mitchell} and R.W. Hartley \cite{hartley}, for $q$ odd and $q$ even, respectively. For an exposition in standard notation,  see  \cite{king}. For our purpose, we limit ourselves to recall the classification for those subgroups which are not isomorphic to a subgroup of  $\PGL(3,q')$ for $q'<q$.

For $q$ odd:
\begin{enumerate}
\item[{\rm (i)}] $\PSL(3,q)$ for $q\equiv 1 \pmod 3$, having order $\frac{1}{3}(q^2+q+1)q^3(q+1)(q-1)^2$
\item[{\rm(ii)}] the stabilizer of a point of $\PG(2,q)$, having order $q^3(q+1)(q-1)^2$
\item[{\rm(iii)}] the stabilizer of a line of $\PG(2,q)$,  having  order $q^3(q+1)(q-1)^2$
\item[{\rm(iv)}] the stabilizer of an Hermitian curve of $\PG(2,q)$ for $q=n^2$,  having  order $n^3(n^3+1)(n-1)^2$
\item[{\rm (v)}] the stabilizer of a triangle of $\PG(2,q)$, having order $6(q-1)^2$
\item[{\rm (vi)}] the stabilizer of an imaginary triangle (i.e., a triangle in $\PG(2,q^3)\setminus \PG(2,q)$), having order $3(q^2+q+1)$
\item[{\rm (vii)}] the stabilizer of an irreducible conic, having order $q(q+1)(q-1)$
\item[{\rm (viii)}] sporadic subgroups: the Hessian groups of order $216$ if $q -1$ is divisible by $9, 72$ and $36$ if $q -1$ is divisible by $3$. Groups of order $168$ (when $-7$ is a square in $\fq$), $360$ (when $5$ is a square in $\fq$) and there is a non-trivial cube root of unity in $\fq$, $720$ (when $q$ is an even
power of $5$), and $2520$ (when $p=5$ and $q$ is an even power of $5$).
\end{enumerate}
Notice that (ii) and (iii) are isomorphic but not conjugate in $\PGL(3,q)$. Up to conjugacy in $\PGL(3,q)$ (equivalently, up to a change of projective frame in $\PG(2,q)$), (ii) is the affine group $\AGL(2,q)$, (iii) is the dual affine group $\overline{\AGL}(2,q)$, (iv) is the projective unitary group $\PGU(3,n)$. Also, (v) is a semidirect product of $C_{q-1}\times C_{q-1}$ by a complement isomorphic to ${\rm{Sym}}_3$,
(vi) is the normalizer of a Singer subgroup of $\PG(2,q)$ of order $q^2+q+1$, and it is a semidirect product of $C_{q^2+q+1}$ by a complement $C_3$.
The groups of orders $168,360,720$, and $2520$ are isomorphic to $\PSL(2,7)$), $\rm{Alt}_6$, $\rm{Alt}_6\cdot 2$
and $\rm{Alt}_7$ respectively. The groups of orders $216$ and $72$ are isomorphic to $\PGU(3,2)$  and $\PSU(3,2)$, respectively, with the
group of order $36$ a subgroup of the latter.

For $q$ even, case (vii) does not occur, while there is a unique sporadic subgroup, namely a group of order $360$ for $q=4$ isomorphic to ${\rm{Alt}}_6$.
\section{Summary of the results}
\label{rr}
The most important family of $\PGL(3,q)$-invariant irreducible plane curves consists of the double-Frobenius nonclassical planes curves $\cF_{n,m}$, introduced and   investigated in \cite{bor}, where ${\rm{gcd}}(m,n)=1$ and $\cF_{n,m}$ stands for the irreducible plane curve of affine equation
\begin{equation}\label{eq:dfbc}
 \frac{(X^{q^n}-X)(Y^{q^m}-Y)-(Y^{q^n}-Y)(X^{q^m}-X)}{(X^{q^2}-X)(Y^q-Y)-(Y^{q^2}-Y)(X^q-X)} = 0
\end{equation}
of genus $g(\cF_{n,m}) = (q^{n-m} + q^m)(\ha q^n-(1 + q + q^2)+ (q + 1)(1 + q + q^2)$.
The special case $\cF_{3,1}$, named DGZ curve in \cite{DGZ}, is widely known and of particular interest as being, for the case $q=p$, the only known example, up to birational equivalency, of an ordinary curve with more than $cg^{3/2}$ automorphisms, for some constant $c$ independent from $g$ and $p$.  This conforms  with the conjecture made by  Guralnick and Zieve in 1998  that no other such examples existed,  see \cite{DGZ}.
The dual curve of $\cF_{3,1}$ is the curve $\cF_{3,2}$, that is, the dual DGZ curve. From Theorem \ref{the22set21A}, $d(\PGL(3,q))=q^3-q$, $\varepsilon(\PGL(3,q))=q^2-q$, and  the DGZ curve is the only  $\PGL(3,q)$-invariant irreducible plane curve of degree $q^3-q$. If $q\equiv 1 \pmod 3$,  then $\PSL(3,q)$ is a maximal subgroup of $\PGL(3,q)$ of index $3$ but the same result holds, see Theorem \ref{the22set22A}.

The subgroup of $\PGL(3,q)$ preserving a line of $\PG(2,q)$ is a maximal subgroup isomorphic to the affine group $\AGL(2,q)$ of order $q^2(q-1)(q^3-q)$. From Theorems \ref{the030921}, \ref{the040921}, $d(\AGL(2,q))=q^3-q^2$, $\varepsilon(\AGL(2,q))=q^2-q$. Furthermore, from Theorem \ref{the050921}, all $\AGL(2,q)$-invariant irreducible curves of degree $q^3-q^2$ belong, up to projectivity, to the pencil
\begin{equation*}
 \frac{(X^{q^3}-X)(Y^{q}-Y)-(Y^{q^3}-Y)(X^{q}-X)}{(X^{q^2}-X)(Y^q-Y)-(Y^{q^2}-Y)(X^q-X)}-\lambda = 0, \quad \lambda \in K.
 \end{equation*}
The subgroup $\overline{\AGL}(2,q)$ of $\PGL(3,q)$ fixing a point of $\PG(2,q)$ is also a maximal subgroup of $\PGL(3,q)$, which is isomorphic (but not conjugate) to $\AGL(2,q)$.  From Theorems \ref{the070921} and \ref{the070921A}, $d(\overline{\AGL}(2,q))=q^3-q^2=d(\AGL(2,q))$ and
$\varepsilon(\overline{\AGL}(2,q)=q^2-q=\varepsilon(\AGL(2,q))$. Furthermore, from Theorem \ref{the070921B}, all $\AGL(2,q)$-invariant irreducible curves of degree $q^3-q^2$ belong, up to projectivity, to the pencil:
\begin{equation*}
 \frac{(X^{q^3}-X)(Y^{q}-Y)-(Y^{q^3}-Y)(X^{q}-X)}{(X^{q^2}-X)(Y^q-Y)-(Y^{q^2}-Y)(X^q-X)}-\lambda \frac{(X^{q^2}-X)(Y^q-Y)-(Y^{q^2}-Y)(X^q-X)}{(Y^q-Y)^{q+1}}= 0,
\end{equation*}
with $\lambda\in K\setminus\{1\}.$

For $q$ square, another large maximal subgroup of $\PGL(3,q)$ is the three-dimensional projective unitary group $\PGU(3,\mathbb{F}_n)$ with $q=n^2$. The best known example of plane curves of $\PG(2,q)$  whose automorphism group is $\PGU(3,n)$ is the Hermitian curve
$\cH_n$, which is the nonsingular curve given by the affine equation
\begin{equation}\label{eq:herm}
 Y^n+Y-X^{n+1} = 0.
\end{equation}
As well known, $\cH_n$ is the unique curve of genus $g$ with more than $16g^4$ automorphisms,  see \cite{stherm}. From Theorem \ref{thm:pgu}, $d(\PGU(3,n))=n+1$, $\varepsilon(\PGU(3,n))=qn-n$, and $\cH_q$ is the only  $\PGU(3,n)$-invariant irreducible plane curve of degree $n+1$.

Further maximal subgroup of $\PGL(3,q)$ is a stabilizer $\Delta_q$ of a triangle, which is the semidirect product of an abelian group of order $(q-1)^2$ fixing the vertices of the triangle by a complement of order $6$ which acts on the set of the vertices as ${\rm{Sym}}_3$ on three letters. From Theorem \ref{the080921C}, $d(\Delta_q)=q-1$, $\varepsilon(\Delta_q)=q-1$, and the only $\Delta_q$-invariant irreducible plane curve of degree $q-1$ is the Fermat curve. Moreover, those that are $\Delta_q$-invariant, and of degree $2q-2$, have homogeneous equation $\lambda(X^{q-1}+Y^{q-1}+Z^{q-1})^2+(XY)^{q-1}+(YZ)^{q-1}+(ZX)^{q-1})=0$ with $\lambda\in K.$

Another maximal subgroup of $\PGL(3,q)$ preserves a triangle, but this time in $\PG(2,q^3)$. It is the normalizer $NS_{q}$ of the Singer subgroup that  is the semi-direct product of a cyclic group of order $q^2+q+1$ (the Singer subgroup) by a complement $C_3$ of order $3$. All irreducible plane curves of degree $d$ that are invariant by the Singer subgroup have been determined for $d\le 2q+2$,  see Lemma \ref{lem110921A} and Theorem \ref{the130921}, which
states that $d(NS_q)=q+2$, $\varepsilon(NS_q)=q-1$, and that the Pellikaan curve of homogeneous equation
\begin{equation}
\label{joco}
 X^{q+1}Y+Y^{q+1}Z+Z^{q+1}X=0
\end{equation}
is the only $NS_q$-invariant curve of degree $q+2$, up to projectivity,  see \cite{Pel, CS1, CS2}.


For $q$ odd, the list of nonsporadic maximal subgroups of $\PGL(3,q)$ also includes the two dimensional projective general linear group $\PGL(2,\fq)$. Several examples of irreducible plane curves whose automorphism group is $\PGL(2,\fq)$ are presented in the literature. Apart from the unique $\PGL(3,q)$-invariant conic, those of minimum degree $q+1$  belong, up to projectivity, to the pencil consisting of the plane curves of homogeneous equation
\begin{equation}\label{eq0}
Y^{q+1}-(X^{q}Z+XZ^q)- \lambda (Y^{2}-2XZ)^{(q+1)/2}=0, \quad \lambda \in K\cup \{\infty\}.
\end{equation}
The above quoted theorems 
provide a complete solution for Questions (i), (ii), and  (iii) when $G$ is one of the maximal, nonsporadic subgroups of $\PGL(3,q)$, namely,  $\PGL(3,q), \PSL(3,q), q\equiv 1 \pmod 3, \AGL(2,q), \overline{\AGL}(2,q),$ $\PGU(3,q), \PGL(2,q), NS_q$,  and $\Delta_q$. In particular, the solution of Question (iii) proves  the existence of previously unknown infinite families of $G$-invariant curves of degree $d(G)$. Moreover, as a byproduct of our investigation, Theorem \ref{thm:pointsdgz} gives the number of points of the DGZ curve over some extensions $\mathbb{F}_{q^i}$ of $\fq$, and analogous results can be obtained for other curves $\cF_{n,m}$.

In most examples of plane curves with large linear automorphism groups which we show in this paper, the whole automorphism group of the curve is linear. Although this appears to be a general behavior, we point out
in Section \ref{hemis} that the opposite case can also occur for some irreducible plane curve $\cC$, that is, the full automorphism group of $\cC$ is large and nonlinear, but its linear automorphisms are still numerous. More precisely, we will show that this happens to the irreducible plane curve $\cK$ of genus $g=q^2-q-1$ and affine equation
\begin{equation}
\label{eq9oct21}
(X+Y)^{q+1} = 2(XY+(XY)^q),
\end{equation}
whose full automorphism group is isomorphic to $\PGL(2,q)$. In fact, the linear automorphisms of $\cK$ form a subgroup of order $q(q-1)>g$.  This curve was encountered in the study of hemisystems of the Hermitian surface arising from the Fuhrmann-Torres $\mathbb{F}_{q^2}$-maximal curves,  see \cite{KNS2017}. Actually, $\cK$ belongs to a larger family of curves with the same behavior. The family is found and investigated with an approach similar to that used in the previous sections. It turns out that one may obtain results  analogous to those therein when investigating curves that are invariant under some finite subgroup of  the  group of birational automorphisms of the projective plane $\PG(2,K)$, the Cremona group.

\section{Proofs of the Results}\label{herm+aut}
\subsection{$\AGL(2,q)$-invariant curves}
\label{ssagl} It is useful to take the line $\ell_\infty$ of equation $Z=0$ and identify the  stabilizer of $\ell_\infty$ in $\PGL(3,q)$ with $\AGL(2,\mathbb{F}_q)=\AGL(2,q)$,  which is the semidirect product of the translation group $T$ by the two-dimensional linear group $\GL(2,q)$. Therefore, for every $\alpha\in \AGL(2,q)$, there exists a $3\times 3$ nonsingular matrix $A$ with all entries in $\mathbb{F}_q$  such that if $L'=a'X+b'Y+c'Z=0$ is  the image of the line $L=aX+bY+cZ=0$ by $\alpha$,  then the vector $[a',b',c']$ is the product of $A$ by the vector $[a,b,c]$. Accordingly, the approach at the onset  of Section \ref{secpgl3} can be used. For this purpose, let $L_i(X,Y,Z)=0$, $i=1,\ldots, q^2+q,$ be the lines of $\PG(2,q)$ different from $\ell_\infty$. Since $\alpha$ preserves the set of all lines $L_i$, the homogeneous polynomial
$$U(X,Y,Z)=\prod\limits_{i=1}^{q^2+q}L_i(X,Y,Z)$$
has the property that $\alpha(U(X,Y,Z))=\det(A)^{q^2+q}U(X,Y,Z)$. As $\det(A)\in \mathbb{F}_q^*$, the $(q-1)$-th power of the product of the $q^2+q$ linear forms $L_i$:
$$G(X,Y,Z)=\big(\prod\limits_{i=1}^{q^2+q}L_i(X,Y,Z)\big)^{q-1}$$
is $\AGL(2,q)$-invariant homogeneous polynomial. Observe that $\deg(G(X,Y,Z))=q^3-q$. Let $\cG$ be the reducible plane curve of equation $G(X,Y,Z)=0$.

Now take the pencil $\Lambda$ generated by $\cG$ together with the line $\ell_\infty$ taken $q^3-q$ times. Then $\Lambda$ consists of all curves $\cC_\lambda$ of equation $ \lambda G(X,Y,Z)-Z^{q^3-q}=0$ where $\lambda\in K \cup \{\infty\}$ and $\cG=\cC_\infty$. Since $\ell_\infty$ is $\AGL(2,q)$-invariant, $\Lambda$ is an $\AGL(2,q)$-fixed pencil.

\begin{lem}\label{pencilB}  Every curve in $\Lambda$ is  reducible, but none of the irreducible components of a curve in $\Lambda$ is $\AGL(2,q)$-invariant.
 \end{lem}
\begin{proof} 
Let $\cC_\lambda$ with $\lambda\in K$ be a curve in $\Lambda$. Define $\mu$ to be  $\mu^{q-1}=\lambda$.  Then $\cC_\lambda$ has affine equation
$$\left(\mu \begin{vmatrix} X & X^{q} & X^{q^2} \\ Y & Y^{q} & Y^{q^2} \\ 1 & 1 & 1 \end{vmatrix}\right)^{q-1}-1=0$$
which can also be written as
$$\prod_{\kappa\in \mathbb{F}_q^*}\left(\mu \begin{vmatrix} X & X^{q} & X^{q^2} \\ Y & Y^{q} & Y^{q^2} \\ 1 & 1 & 1 \end{vmatrix}-\kappa\right)=0.$$
This shows that $\cC_\lambda$ splits into $q-1$ plane curves $\cF_\kappa$ with $\kappa \in \mathbb{F}_q^*$. The projectivity $\pi_\rho: (X,Y)\mapsto (\rho X,  Y)$ with $\rho \in \mathbb{F}_q^*$ is in $\AGL(2,q)$ and
takes $\cF_\kappa$ to $\cF_{\rho^{-1}\kappa}$, and the subgroup $R=\{\pi_\rho | \rho\in \mathbb{F}_q^*\}$ acts transitively on the set  of the set of the $q-1$ plane curves $\cF_\kappa$. Therefore, none of these curves is $\AGL(2,q)$-invariant. From this, our claim follows.
\end{proof}
\begin{lem}\label{lem4set21}  No irreducible plane curve of degree smaller than $q^3-q^2$ is  $\AGL(2,q)$-invariant.
\end{lem}
\begin{proof} Let $\cD$ be an $\AGL(2,q)$-invariant irreducible plane curve, and  select  a point $P\in \cC$ from a long $\AGL(2,q)$-orbit that is not a base point of $\Lambda$. Let $\cC_\lambda$ be the unique curve in $\Lambda$ through $P$. Then $\cD$ is an irreducible component of $\cC_\lambda$; otherwise the B\'ezout theorem would yield $\deg(\cD)\ge |\AGL(2,q)|/\deg(\cC_\lambda)=q^2(q-1)(q^3-q)/(q^3-q)=q^3-q^2$. Therefore,  the claim follows from Lemma \ref{pencilB}.
\end{proof}
Since the DGZ curve is an $\AGL(2,q)$-invariant curve of degree $q^3-q^2$,  Lemma \ref{lem4set21} solves Question (i) for $G=\AGL(2,q)$.
\begin{theorem}
\label{the030921} $d(\AGL(2,q))=q^3-q^2$.
\end{theorem}
Now we are in a position to solve Question (ii) for $G=\AGL(2,q)$. Since the dual DGZ curve is an $\AGL(2,q)$-invariant curve of degree $q^3-q$, we have $\varepsilon \ge q^3-q-(q^3-q^2)=q^2-q$. We prove that equality holds.
\begin{theorem}
\label{the040921}  $\varepsilon(\AGL(2,q))=q^2-q$.
\end{theorem}
\begin{proof}
We use the $\AGL(2,q)$-fixed pencil $\Lambda$ generated by the DGZ curve $\cF_{3,1}$ and the totally reducible curve $\cL$ of equation $Z^{q^3-q^2}$ which is the line $\ell_\infty$ taken $q^3-q^2$.  Let $\cD$ be an $\AGL(2,q)$-invariant irreducible plane curve whose degree is smaller than $q^3-q$. Then the proof of lemma \ref{lem4set21} shows that $\cD$ is a component of a curve $\cC_\lambda$ in $\Lambda$. There is component $\cF$ of $\cC_\lambda$ such that $\cC_\lambda$ splits into $\cD$ and $\cF$. Since both $\cD$ and $\cC_\lambda$ are $\AGL(2,q)$-invariant, $\cF$ is also $\AGL(2,q)$-invariant. From Lemma \ref{pencilB}, $\deg(\cD)\ge q^3-q^2$. Therefore, $\deg(\cF)< q^2-q$, and $\cF$ is reducible by Lemma \ref{pencilB}. We show that no line of $\AG(2,q)$ is a component of $\cF$. If this occurs for a line then every line of $\AG(2,q)$ is a component of $\cF$ as $\AGL(2,q)$ acts transitively on the set of all lines of $\AG(2,q)$. Since we have $q^2+q$ such lines this would yield $\deg(\cF)\ge q^2+q$, a contradiction. Moreover, $\ell_\infty$ is not a component of $\cC_\lambda$ since $\cL$ is the unique curve in $\Lambda$ containing $\ell_\infty$ as a component whereas $\cD$ is an irreducible plane curve of $\deg(\cD)>1$.

Now take a line $\ell$ of $\AG(2,q)$. The stabilizer of $\ell$ in $\AGL(2,q)$ has order $q^2(q-1)(q^3-q)/(q^2+q)=(q-1)^2q$ and induces on $\ell$ a permutation group $\Sigma$ of order $q(q-1)$. Since $\Sigma$ is an automorphism of $\ell$, it has one fixed point $P_\infty=\ell\cap \ell_\infty$ and one nontrivial short orbit consisting of all points of $\ell\cap \AG(2,\mathbb{F}_q)$ other than $P_\infty$, while the other $\Sigma$-orbits are long of length $q^2-q$. Since $\deg(\cF)<q^2-q$ and $\ell$ is not a component of $\cF$, the B\'ezout theorem yields that one of the following cases occurs
 $$\cF\cap \ell=\begin{cases}
 \ell\cap \ell_\infty=\{P_\infty\},\\
 \ell\cap\AG(2,q),\\
 \ell\cap \PG(2,q).
 \end{cases}
 $$
Observe that when one of the above cases occurs for line $\ell$ then  the same case does for all lines of $\AG(2,q)$, as the stabilizer of $P_\infty$ in $\AGL(2,q)$ acts transitively on the set of the $q$ lines of $\AG(2,q)$ passing through $P_\infty$, and $\AGL(2,q)$ acts transitively on the points of $\ell_\infty\cap \AG(2,q)$.

In the first case, for every line of $\AG(2,q)$ passing through $P_\infty$ we have $I(P_\infty, \cF\cap \ell)=\deg(\cF)$.  Let $\mu(P_\infty)\ge 1$ be the multiplicity of $P_\infty$ as a point of $\cF$. Then $\mu(P_\infty)\le \deg \cF$ and if equality holds then $\cF$ splits into lines through $P_\infty$ whereas we have already shown that no line of $\PG(2,q)$ is a component of $\cF$. Therefore $\mu(\cF)<deg(\cF)$, and hence of the $q$ lines of $\AG(2, q)$ passing through $P_\infty$ is a tangent to $\cF$ at $P_\infty$. Therefore, $\mu(P_\infty)\geq q$. Since this holds true for every point in $\ell_\infty \cap \AG(2,q)$, the B\'ezout theorem yields that $\deg(\cF)\geq q(q+1)>q^2-q$, a contradiction.

In the second and third case, for every point $P\in \AG(2,q)$, either $\mu(P)\ge q+1$, or each line of $\AG(2,q)$ passing through $P$ is a tangent to $\cF$ at $P$. Therefore, $\cF$ has $q$ singular points of multiplicity $\ge q$ on any line of $\AG(2,q)$ passing through $P$. As we have already shown,  $\ell$ is not a component of $\cF$. Therefore, the B\'ezout theorem yields $\deg(\cF)\ge q^2>q^2-q$, a contradiction.
\end{proof}
From Theorems \ref{the030921} and \ref{the040921}, we also obtain a complete solution of Question(iii).
\begin{theorem}
\label{the050921} The $\AGL(2,q)$-invariant irreducible plane curves of degree $q^3-q^2=d(\AGL(2,q))$ are those of affine equation
\begin{equation}\label{eq:dfbcA}
 \frac{(X^{q^3}-X)(Y^{q}-Y)-(Y^{q^3}-Y)(X^{q}-X)}{(X^{q^2}-X)(Y^q-Y)-(Y^{q^2}-Y)(X^q-X)}-\lambda = 0.
\end{equation}
with $\lambda \in K$.
\end{theorem}
\subsection{$\overline{\AGL}(2,q)$-invariant curves}
\label{ssaglbis}
We will see that the main results stated in Section \ref{ssagl} hold true for $\overline{\AGL}(2,q)$-invariant curves. It should be noted however that one cannot simply obtain such results by dual arguments since the dual curve of an irreducible plane curve may not exist, as we work in positive characteristic where nonclassical plane curves can occur, and even in the case when the dual curve exists its degree is different, in general.

Let $P_0$ be the unique fixed point of $\overline{\AGL}(2,q)$. Then $\overline{\AGL}(2,q)$ induces a permutation group on the lines through $P_0$, and the kernel $N$ of this permutation representation is the semidirect product of en elementary abelian group $T$ of order $q^2$ (the dual translation group with center $P)$ by a cyclic complement $C_{q-1}$ of order $q-1$.
\begin{lemma}
\label{lem050921}  If a plane curve $\cC$ of degree smaller than $q^3-q^2$ is $\overline{\AGL}(2,q)$-invariant then it splits into lines through $P_0$.
\end{lemma}
\begin{proof} Take any point $Q$ of $\cC$ which is not fixed by any non-trivial element of $N$. Then the
$N$-orbit of $Q$ consists of $q^2(q-1)$ collinear points on a line $\ell$ through $P$ each of them lying on $\cC$. Since $\deg(\cC)<q^3-q^2=q^2(q-1)$, this implies that $\ell$ is a component of $\cC$.  Moreover, the number of fixed points of elements in $N$ is finite. Therefore, there are infinitely many lines $\ell$ through $P_0$ such that $I(P_0,\cC\cap \ell)=\deg(\cC)$. This implies that $P_0$ is a singular point of $\cC$ of multiplicity $\deg(\cC)$. Thus, $\cC$ splits into lines through $P_0$.
\end{proof}
Lemma \ref{lem050921} solves Question (i) since the DGZ curve is an $\overline{\AGL}(2,q)$-invariant curve of degree $q^3-q^2$.
\begin{theorem}
\label{the070921} $d(\overline{\AGL}(2,q))=q^3-q^2$.
\end{theorem}
\noindent
Question(ii) also has the same solution as for $\AGL(2,q)$.
\begin{theorem}
\label{the070921A}  $\varepsilon(\overline{\AGL}(2,q))=q^2-q$.
\end{theorem}
\begin{proof} Fix a reference system $(X,Y)$ in the plane $\AG(2,q)$ such that $P_0$ be the infinite point $(1,0,0)$ of the $X$-axis. Then the lines in $\AG(2,\mathbb{F}_q)$ not incident with $P_0$ have homogeneous equation  $X+aY+bZ=0$ with $a,b\in \mathbb{F}_q$. Therefore, the curve $\cL$ whose components are the lines through $P_0$
has homogeneous equation $(Y^{q-1}-Z^{q-1})YZ=0$.

Thus, the totally reducible plane curve $\cF$  whose component are exactly 
the lines  off $P_0$ each taken $q-1$ times has affine equation
 $$\frac{\left(\begin{vmatrix} X & X^{q} & X^{q^2} \\ Y & Y^{q} & Y^{q^2} \\ Z & Z^q & Z^{q^2} \end{vmatrix}\right)^{q-1}}{((Y^{q-1}-Z^{q-1})YZ)^{q-1}}=0.$$
In particular, 
$ \mathcal{F}$ is an $\overline{\AGL}(2,q)$-invariant curve of degree $q^3-q^2$.
In fact, if $\alpha\in \overline{\AGL}(2,q)$ and $[X,Y,Z]=A[\bar{X},\bar{Y},\bar{Z}]$ with a regular matrix $A=(a_{ij})$ then $X=a_{11}\bar{X}$, $Y=a_{22}\bar{Y}+a_{23}\bar{Z}$ and $Z=a_{32}\bar{Y}+a_{33}\bar{Z}$. By a straightforward computation $(Y^{q-1}-Z^{q-1})YZ= \det(A)(\bar{Y}^{q-1}-\bar{Z}^{q-1})\bar{Y}\bar{Z}$  whence the claim follows.
This time we use the $\overline{\AGL}(2,q)$-fixed pencil $\Lambda$ generated by 
$\mathcal{F}$ and the DGZ curve $\cF_{3,1}$.  Let $\cD$ be an $\overline{\AGL}(2,q)$-invariant irreducible plane curve whose degree is smaller than $q^3-q$. From the proof of Lemma \ref{lem4set21}, $\cD$ is a component of a curve $\cC_\lambda$ in $\Lambda$. Therefore, there is a component $\cF$ of $\cC_\lambda$ such that $\cC_\lambda$ splits into $\cD$ and $\cF$. Since both $\cD$ and $\cC_\lambda$ are $\overline{\AGL}(2,q)$-invariant, $\cF$ is also $\overline{\AGL}(2,q)$-invariant. Moreover,
$\deg(\cF)<q^2-q$. From Lemma \ref{lem050921}, $\cF$ splits into lines through $P_0$. If one of these lines was not defined over $\mathbb{F}_q$, then that line would have at least $q^2-q$ images under the action of $\overline{\AGL}(2,q)$. But this would imply $\deg(\cF)\ge q^2-q$ , a contradiction. Since $\overline{\AGL}(2,q)$ acts transitively on the set of the $q+1$ lines of $\PG(2,q)$ passing through $P_0$, each such line is a component of $\cF$ with same multiplicity $r$. Thus, $\cF$ has homogeneous equation $(\prod_{\omega\in \mathbb{F}_q} (Y-\omega Z)Z)^r=0$. Since $\cF$ is $\overline{\AGL}(2,q)$-invariant, $q-1$ divides $r$. But this implies
$\deg(\cF)\ge q^2-q$, a contradiction.
\end{proof}
From Theorems \ref{the070921} and \ref{the070921A}, we also obtain a complete solution of Question(iii).
\begin{theorem}
\label{the070921B} The $\overline{\AGL}(2,q)$-invariant irreducible plane curves of degree $q^3-q^2=d(\AGL(2,q))$ are those of affine equation
\begin{equation*}\label{eq:dfbcB}
 \frac{(X^{q^3}-X)(Y^{q}-Y)-(Y^{q^3}-Y)(X^{q}-X)}{(X^{q^2}-X)(Y^q-Y)-(Y^{q^2}-Y)(X^q-X)}-\lambda \frac{(X^{q^2}-X)(Y^q-Y)-(Y^{q^2}-Y)(X^q-X)}{(Y^q-Y)^{q+1}}= 0.
\end{equation*}
with $\lambda \in K\setminus \{1\}$.
\end{theorem}

\subsection{$\PGL(3,q)$-invariant curves}
\label{secpgl3} In Section \ref{rr} we have exhibited an infinite family of $\PGL(3,q)$-invariant plane curves containing the DGZ curve of degree $q^3-q^2$ and the dual DGZ curve of degree $q^3-q$. It may be noticed that the latter two curves are of the lowest degree in that family since the others have degree at least $q^4$.

We begin by constructing a $\PGL(3,q)$-fixed pencil. Then, we show that the curves in this pencil have degree $q^4-q$ apart from a few sporadic cases, including the DGZ and the dual DGZ curves.

Let $L_i(X,Y,Z)=0$, $i=1,\ldots, q^2+q+1 $ be the lines of $\PG(2,q)$, and  let  $\tilde{L}_j(X,Y,Z)=0$, $j=1,\ldots,q^4-q$  be the lines of $\PG(2,q^2)$ which are not defined over $\fq$.
For $$G(X,Y,Z)= \prod\limits_{j=1}^{q^4-q}\tilde{L}_j,\quad H(X,Y,Z)=  \prod\limits_{i=1}^{q^2+q+1} L_i$$ consider the pencil $\Lambda$ generated by the plane curves of homogeneous equations $G(X,Y,Z)=0$ and $H(X,Y,Z)^{q(q-1)}=0$, respectively. The following lemma collects some useful properties of $\Lambda$.
\begin{lem}\label{pencil} The pencil $\Lambda$ consists of all curves $\cC_\lambda$ with $\lambda \in K\cup \{\infty\}$ of homogeneous equation  $G(X,Y,Z)- \lambda H(X,Y,Z)^{q(q-1)}=0$.
Furthermore,
  \begin{enumerate}[\rm(i)]
  \item  $\Lambda$ is a $\PGL(3,q)$-fixed pencil.
  \item  each $\mathcal{C}_\lambda$ contains all points in $\PG(2,q^2)$.
  \item  each point $P\in \PG(2,q)$ is an ordinary singularity of $\mathcal{C}_\lambda$ with $\lambda \in K$, the tangent lines to $\cC_\lambda$ at $P$ being the $q^2-q$ lines through $P$ which are defined over $\fqq$ but over $\fq$.
  \item  every line $\ell$ of $\PG(2,q)$ meets $\mathcal{C}_\lambda$ with $\lambda \in K$  in the points of $\ell$ over $\fqq$, and
  \begin{equation*}
I\left(P, \ell \cap \mathcal{C}_\lambda\right)=
\begin{cases}
q^2-q, & \text{ if } P \in \PG(2,q)\\
q^2, & \text{ if } P \in \PG(2,q^2)\setminus  \PG(2,q).
\end{cases}
\end{equation*}
  \item the components of $\mathcal{C}_1$  are the DGZ curve  $\mathcal{F}_{3,1}$ taken $q$ times, and the dual DGZ curve $\mathcal{F}_{3,2}$.
  \end{enumerate}
    \end{lem}
    \begin{proof} (i) For a projectivity $\alpha\in \PGL(3,q)$ associated with a regular $3\times 3$-matrix $A=(a_{ij})$, let $[\bar{X},\bar{Y},\bar{Z}]$ be the vector such that $[X,Y,Z]$ is the product of $A$ by $[\bar{X},\bar{Y},\bar{Z}]$. Since
$$H(X,Y,Z)=\begin{vmatrix} X & X^{q} & X^{q^2} \\ Y & Y^{q} & Y^{q^2} \\ Z & Z^q & Z^{q^2} \end{vmatrix},$$
we have $H(X,Y,Z)=\det(A) H(\bar{X},\bar{Y},\bar{Z})$. Furthermore, $\det(A)\in \fq^*$ yields $\det(A)^{q-1}=1$. Thus $H(X,Y,Z)^{q(q-1)}=H(\bar{X},\bar{Y},\bar{Z})^{q(q-1)}$. Similarly, since
$$G(X,Y,Z)H(X,Y,Z)=\begin{vmatrix} X & X^{q^2} & X^{q^4} \\ Y & Y^{q^2} & Y^{q^4} \\ Z & Z^{q^2} & Z^{q^4} \end{vmatrix}$$
we have $G(X,Y,Z)H(X,Y,Z)=\det(A) G(\bar{X},\bar{Y},\bar{Z})H(\bar{X},\bar{Y},\bar{Z})$. This together with $H(X,Y,Z)=\det(A)H(\bar{X},\bar{Y},\bar{Z})$ yield $G(X,Y,Z)=G(\bar{X},\bar{Y},\bar{Z})$. Therefore,
$\Lambda$ is a $\PGL(3,q)$-fixed pencil.

(ii) Since both generators of $\Lambda$ contain all points of $\PG(2,q)$, each of such points is a based point of $\Lambda$. Thus, they are common points of the curves in $\Lambda$.

(iii) The claim clearly holds for $\cC_0$. Moreover, since any point $P\in \PG(2,q)$ is incident with exactly $q+1$ lines in $\PG(2,q)$, every such point $P$ is a point of $C_{\infty}$ with multiplicity $(q+1)q(q-1)>q^2-q$. Therefore, the claim holds true in $\Lambda$ for every $\cC_\lambda$, with the unique exception for $\lambda=\infty$.

 (iv) By (ii) and (iii), it remains to show that $I\left(P, \ell \cap \mathcal{C}_\lambda\right)=q^2$ for $P \in \PG(2,q^2)\setminus  \PG(2,q)$. We first do it for $\cC_0$.
 Through $P$ there are exactly $q^2+1$ lines in $\PG(2,q^2)$, just one of them, say $\ell$, is defined over $\fq$ and hence is not a component of $\cC_0$. Thus $I(P, \ell \cap \mathcal{C}_0)=\sum_{i=1}^{q^2}I(P, \ell \cap \tilde{L}_i)=q^2.$ Since $\ell$ is a component of $\cC_\infty$ this, holds true for any curve in $\Lambda$ for $\lambda \in K$.

 (v) By direct inspection.

 \end{proof}
\begin{theorem}
\label{the22set21} Every $\PGL(3,q)$-invariant irreducible plane curve of degree smaller than $q^4-q^2$ is a component of a unique curve in $\Lambda$.
\end{theorem}
\begin{proof} We use the argument of the proof of Lemma \ref{lem4set21}. Let $\cD$ be an $\PGL(3,q)$-invariant irreducible plane curve, and choose a point $P\in \cC$ from a long $\PGL(3,q)$-orbit which is not a base point of $\Lambda$. Let $\cC_\lambda$ be the unique curve in $\Lambda$ through $P$. Then $\cD$ is an irreducible component of $\cC_\lambda$, otherwise the B\'ezout theorem would yield $\deg(\cD)\ge |\PGL(3,q)|/\deg(\cC_\lambda)=q^2(q-1)(q^2+q+1)(q^3-q)/(q^4-q)=q^4-q^2$. 
\end{proof}
\begin{theorem} \label{thm:dgz}
     The DGZ curve is the only irreducible plane curve of degree $d<q^3-q$ which is $\PGL(3,q)$-invariant.
     \end{theorem}
        \begin{proof}
 Let $\cC$ be an irreducible plane curve of degree $d<q^3-q$ which is $\PGL(3,q)$-invariant.       Since $\overline{\AGL}(2,q)$ is a subgroup of $\PGL(3,q)$, Theorems \ref{the070921} and \ref{the070921A} yield that $\deg(\cC)=q^3-q^2$ and that $\cC$ belongs to the $\overline{\AGL}(2,q)$-fixed pencil $\Lambda$ defined in the proof of Theorem \ref{the070921A}. Assume that $\cC$ is not the DGZ curve, that is, $\cC\neq \cF_{3,1}$. Then that pencil $\Lambda$ is generated by $\cC$ and $\cF_{3,1}$. Since this holds true for any subgroup of $\PGL(3,q)$ conjugate to $\overline{\AGL}(2,q)$, it turns out that $\Lambda$ is a $\PGL(3,q)$-fixed pencil which contains any reducible plane curve consisting of all the $q+1$ lines of $\PG(2,q)$ through a point in $\PG(2,q)$ taken each ($q-1$)-times. Such a pencil cannot actually exist as $\PG(2,q)$ contains three non collinear points determining three such reducible curves that cannot belong to the same pencil.
 \end{proof}

Since $\AGL(2,q)$ is a subgroup of $\PGL(3,q)$, Theorems \ref{the070921} and \ref{the070921A} together with Theorem \ref{thm:dgz}
give a complete solution for Questions (i),(ii) and (iii).
\begin{theorem}
\label{the22set21A} $d(\PGL(3,q))=q^3-q^2$, $\varepsilon(\PGL(3,q))=q^2-q$, and the DGZ curve is the only $\PGL(3,q)$-invariant irreducible plane curve of degree $d=q^3-q^2$.
\end{theorem}
\subsection{$\PSL(3,q)$-invariant curves}
Our aim is to show that Theorem \ref{the22set21A} holds true for $\PSL(3,q)$-invariant curves. We may assume that $q\equiv 1 \pmod 3$ otherwise $\PGL(3,q)= \PSL(3,q)$. In this subsection, $t$ stands for an $\fq$-rational line while $\ell$ denotes an $\fqq$-rational line which is not $\fq$-rational.
\begin{theorem}
\label{the22set21B} Every $\PSL(3,q)$-invariant irreducible plane curve of degree smaller than $q^3-q$ is also $\PGL(3,q)$-invariant.
\end{theorem}
\begin{proof}
 Up to conjugacy in $\PGL(3,q)$, the stabilizer  of $t$ in $\PSL(3,q)$ is ${\rm{ASL}}(2,q)$. Since the subgroup of ${\rm{ASL}}(2,q)$ fixing $t$ pointwise has order $\frac{1}{3}(q-1)q^2$, ${\rm{ASL}}(2,q)$ acts on $t$ as $\PGL(2,q)$. Therefore, ${\rm{ASL}}(2,q)$ has exactly two short orbits on $t$, namely $t\cap\PG(2,q)$ and $t\cap(\PG(2,q^2)\setminus \PG(2,q))$, having length $q+1$ and $q^2-q$, respectively. Let $\ell$ be a line of $\PG(2,q^2)$ but not of $\PG(2,q)$. Then the stabilizer of $\ell$ in $\PSL(3,q)$ has order $|\PSL(3,q)|/(q^4-q)=\frac{1}{3}q^2(q-1)^2$, and hence it contains a subgroup $S_q$ of order $q^2$. Clearly, $S_q$ fixes the unique common point $P$ of $\ell$ with $\PG(2,q)$. If $S_q$ fixed another point in $\ell$ then it would fix each point of $\ell$ but any line fixed pointwise by a projectivity of order $p$ is of $\PG(2,q)$. Therefore, $S_q$ acts transitively on $\ell\cap \PG(2,q^2)\setminus \{P\}$. Let $Q$ be a point in $\PG(2,q^2)\setminus \PG(2,q)$. Up to conjugacy, the stabilizer $U$ of $Q$ in $\PSL(3,q)$ has order $|\PSL(3,q)|/(q^4-q)=\frac{1}{3}q^2(q-1)^2$. Within the set of the lines in the pencil with center $Q$, the group $U$ preserves the unique line of $\PG(2,q)$, it is transitive on the set of the $q^2$ lines in $\PG(2,q^2)\setminus \PG(2,q)$ and each of the remaining $U$-orbits is long.
Let $\cC$ be a $\PSL(3,q)$-invariant irreducible plane curve which is not $\PGL(3,q)$-invariant.

{\em Case (i)} If some of the common points of $t$ with $\cC$ is not contained in $\PG(2,q^2)$ then the ${\rm{ASL}}(2,q)$-orbit of that point has length $q^3-q$ and hence $\deg(\cC)\ge q^3-q$.

{\em Case (ii)}. Assume that $\cC$ has a point $Q$ in $\PG(2,q^2)\setminus \PG(2,q)$, and let $s$ be a tangent to $\cC$ at $Q$. If $s$ is not a line of $\PG(2,q^2)$ then there exist at least $\frac{1}{3}q^2(q-1)^2$ tangents to $\cC$ at $Q$, and hence $Q$ is a point with multiplicity at least $\frac{1}{3}q^2(q-1)^2 > \deg(\cC)$, a contradiction.
If $s$ is a line of  $\PG(2,q^2)\setminus \PG(2,q)$ then the previous argument shows that $Q$ is at least a $q^2$-fold point of $\cC$, whence  $\deg(\cC)\ge q^2(q^2-q) > q^3-q$ follows by B\'ezout's theorem applied to $(\cC,t)$, where $Q\in t$. It remains to consider the case where $\cC$ has a unique tangent line at $Q$, namely the line $t$ through $Q$.

{\em Case (iia)} Assume that $\cC$ has no point in $\PG(2,q)$. Then $|t\cap \cC|=q^2-q$ and $|\ell\cap \cC|=q^2$. Moreover, let  $\alpha=I(Q,t\cap \cC)$ and $\beta=I(Q,\ell\cap \cC)$. Then  $\deg(\cC)=\alpha (q^2-q)$ and $\deg(\cC)=\beta q^2$ by B\'ezout's theorem applied to $(\cC,t)$ and $(\cC,\ell)$, respectively. Thus, $\alpha=\gamma q$ and $\beta=\delta (q-1)$ for two positive integers $\gamma$ e $\delta$. For $\gamma \ge 2$ (or $\delta\ge 2$), this yields $\deg(\cC)\ge 2(q^3-q^2)>q^3-q$. Otherwise $\gamma=\delta=1$ whence $\deg(\cC)=q^3-q^2$. Also,
$Q$ is a ($q-1$)-fold point of $\cC$ as $\ell$ is not a tangent to $\cC$ at $Q$. Since $\PGL(3,q)\gneqq \PSL(3,q)$, some projectivity in $\PGL(3,q)$ takes $\cC$ to another $\PSL(3,q)$-invariant irreducible plane curve $\cD$. Each point $Q$ in $\PG(2,q^2)\setminus \PG(2,q)$ is a common point of $\cC$ and $\cD$ where they also have the same tangent. Therefore, $I(Q,\cC\cap \cD)\ge (q-1)^2+1$. From B\'ezout's theorem applied to $(\cC,\cD)$, $\deg(\cC)\,\deg(\cD)\ge (q^4-q)((q-1)^2+1)$.  However, since $\deg(\cD)=\deg(\cC)=q^3-q^2$, this would yield $q^4(q-1)^2\ge (q^4-q)((q-1)^2-1)$, a contradiction.

{\em Case (iib)} Assume that $\cC$ has a point $P$ in $\PG(2,q)$, and let $s$ be a tangent to $\cC$ at $P$. Arguing as before, if $s$ is not a line of $\PG(2,q^2)$ then $\frac{1}{3}q^2(q-1)^2\ge q^3-q$. Moreover, if $s$ is a line of $\PG(2,q^2)\setminus \PG(2, q)$ then $P$ is a point of multiplicity at least $q^2-q$. Therefore, $\deg(\cC)>(q+1)(q^2-q)=q^3-q$. If $s$ is a line of $\PG(2,q)$, then $P$ is an $r$-fold point with  $r\ge q+1$. Take a line $t$ of $\PG(2,q)$ passing through $P$. Let $I(P,t\cap \cC)=r+\rho$ with $\rho\ge 1$. For a point $Q\in t$ of $\PG(2,q^2)\setminus \PG(2,q)$, let $I(Q,t\cap \cC)=v+\nu$ where $v$ is the multiplicity of $\cC$ at $Q$.
B\'ezout's theorem applied to $(\cC,t)$ yields $\deg(\cC)=(q^2-q)(v+\nu)+(q+1)(r+\rho)$. On the other hand, if $\ell$ is a line of $\PG(2,q^2)\setminus \PG(2,q)$, B\'ezout's theorem applied to $(\cC,\ell)$ gives $\deg(\cC)=q^2v+r$.
Therefore, $(q^2-q)(v+\nu)+(q+1)(r+\rho)=q^2v+r$, whence $q(r+\rho)+\rho+\nu q^2=(v+\nu)q$. From this, $\rho=\sigma q$ for a positive integer $\sigma$. Thus $v=r+\sigma (q+1)+\nu(q-1)\ge 2q$. It turns out that $\deg(\cC)> vq^2\ge 2q^3>q^3-q$.

{\em Case (iii)}. Assume that the points of $\cC$ in $\PG(2,q^2)$ are those in $\PG(2,q)$. For a point $P\in \PG(2,q)$, let $s$ be a tangent to $\cC$ at $P$. Arguing as before, if $s$ is not a line of $\PG(2,q^2)$, then $\frac{1}{3}q^2(q-1)^2\ge q^3-q$. Moreover, if $s$ is a line of $\PG(2,q^2)\setminus \PG(2,q)$ then $\deg(\cC)\ge (q+1)(q^2-q)=q^3-q$ by B\'ezout's theorem applied to $(\cC,t)$. If each  line of $\PG(2,q)$ passing through $P$ is a tangent to $\cC$ then $P$ is a $r$-fold point of $\cC$ with $r\ge q+1$. On the other hand, if $P\in \ell$ then either $\ell\cap \cC=\{P\}$, or $\ell\cap \cC$ contains a point $R$ off $\PG(2,q^2)$. In the latter case, $\ell\cap \cC$ consists of $P$ together with one or more long point-orbits  under the action of the stabilizer of $\ell$ in $\PSL(3,q)$. As we have already shown, such a long orbit has length at least $(q(q-1))^2$. This yields $\deg(\cC)>(q(q-1))^2+(q+1)>q^3-q$. Finally, if $\ell\cap \cC=\{P\}$ then $I(P,\ell\cap \cC)=\deg(\cC)$, and hence $\ell$ would be a tangent to $\cC$ at $P$, a contradiction.
\end{proof}
Theorems \ref{the22set21A} and \ref{the22set21B} have the following corollary.
\begin{theorem}
\label{the22set22A} $d(\PSL(3,q))=q^3-q^2$, $\varepsilon(\PSL(3,q))=q^2-q$, and the DGZ curve is the only $\PSL(3,q)$-invariant irreducible plane curve of degree $d=q^3-q^2$.
\end{theorem}

\subsection{$\PGU(3,n)$-invariant curves for $q=n^2$}\label{sec:pgu}
Up to a change of reference system in $\PG(2,q)$, the Hermitian curve $\cH_n$ has homogeneous equation  $Y^nZ+YZ^n-X^{n+1}=0$. Moreover, $\cH_n$ is a nonsingular plane curve and its automorphism group is $\PGU(3,n)$ which acts as a doubly transitive permutation group on the set of the $n^3+1$ points of $\cH_n$ in $\PG(2,q)$. Furthermore, $\PGU(3,n)$ is generated by the following  projectivities.
\begin{enumerate}[\rm(i)]
\item  $\alpha_1:(X,Y,Z)\mapsto (X,Z,Y)$,
\item $\alpha_2: (X,Y,Z)\mapsto (X+uZ, u^{n} X +Y+eZ, Z)$, with  $u,e\in \fq$ and  $e^{n}+e=u^{n+1}$,
\item $\alpha_3: (X,Y,Z)\mapsto (cX, c^{n+1} Y, Z)$, with   $c^{q-1}=1$,
\end{enumerate}
We exhibit a $\PGU(3,n)$-fixed pencil. For this purpose, set $F(X,Y,Z)=Y^{n}Z+YZ^n-X^{n+1}$ and $G(X,Y,Z)=Y^{n^3}Z+YZ^{n^3}-X^{n^{3}+1}$. For $i=1,2,3$, let $A_k=(a_{ij})$ be a $3\times 3$ nonsingular matrix associated with $\alpha_k$. Then $F(X,Y,Z)$ is  an $A_k$-invariant polynomial for $k=1,2$, whereas $F(X,Y,Z)=c^{n+1}F(\bar{X},\bar{Y},\bar{Z})$ where $c\in \fq$ and the vector  $[X,Y,Z]$ is the product of $A_3$ by the vector $[\bar{X},\bar{Y},\bar{Z}]$. Therefore, $F(X,Y,Z)^{q-n+1}=c^{n^3+1}F(\bar{X},\bar{Y},\bar{Z})^{q-n+1}$. Since $G(X,Y,Z)=0$ can be viewed as the homogeneous equation of the Hermitian curve $\cH_{n^3}$ over $\mathbb{F}_{n^3}$, the argument also shows that $G(X,Y,Z)$ is both  an $\alpha_1$ and $\alpha_2$ invariant polynomial whereas $G(X,Y,Z)=c^{n^3+1}G(\bar{X},\bar{Y},\bar{Z})$, where $c\in \fq$ as $A_3$ is defined over $\fq$.
Therefore, $\cH_{n^3}$ together with $\cH_n$, taken $q-n+1$ times, generate a $\PGU(3,n)$-fixed pencil $\Lambda$.
\begin{lem}\label{pencil7oct} The pencil $\Lambda$ consists of all curves $\cC_\lambda$ with $\lambda \in K\cup \{\infty\}$ of homogeneous equation  $G(X,Y,Z)- \lambda F(X,Y,Z)^{q-n+1}=0$.
Furthermore,
\begin{enumerate}[\rm(i)]
\item $\Lambda$ is a $\PGU(3,n)$-fixed pencil.
\item $\cC_\lambda$ is nonsingular for $\lambda\neq 1$.
\item $\cC_1$ splits into lines.
\end{enumerate}
\end{lem}
\begin{proof}(ii) Let $V(X,Y,Z)=G(X,Y,Z)- \lambda F(X,Y,Z)^{q-n+1}$ with $\lambda \in K$. Then $P=(x:y:z)$ is a singular point of $\cC_\lambda$ if and only if the partial derivatives $V_X,V_Y,V_Z$ vanish in $(x,y,z)$. A straightforward computation shows that this occurs if and only if
$$x^{n^3}=\lambda(y^nz+yz^n-x^{n+1})^{q-n}x^n,\,\,z^{n^3}=\lambda(y^qn+yz^n-x^{n+1})^{q-n}z^n,\,\, y^{n^3}=\lambda(y^nz+yz^n-x^{n+1})^{q-n}y^n.$$
Clearly, $y^nz+yz^n-x^{n+1}\neq 0$, otherwise $x=y=z=0$ would hold. Suppose first that $x\neq 0$. Then $x=1$ can be assumed. Therefore, $\lambda(y^nz+yz^n-1)^{q-n}=1$ and hence $y^{n^3}=y^n, z^{n^3}=z^n$ yielding $y,z\in \fq$. On the other hand, if $y,z\in \fq$ then $(y^nz+yz^n-1)^n=y^nz+yz^n-1$ and hence $(y^nz+yz^n-1)^{q-n}=1$. From this, $\lambda=1$ follows. A similar argument shows that $\lambda=1$ still holds if either $y\neq 0$ or $z\neq 0$.

 (iii) Clearly, the line $\ell_\infty$ of equation $Z=0$ is a component of $\cC_1$. On the other hand, $\ell_\infty$ is the tangent to $\cH_n$ at the point $Y_\infty=(0:1:0)$. Since $\cH_n$ is $\PGU(3,n)$-invariant, the reducible plane curve of degree $n^3+1$ whose components are the tangents to $\cH_q$ at points in $\PG(2,q)$ is $\PGU(3,n)$-invariant. Actually, this plane curve coincides with $\cC_1$ as $\cC_1$ is  also $\PGU(3,n)$-invariant.
\end{proof}

\begin{theorem}\label{thm:pgu}
Let $\cC$ be a $\PGU(3,n)$-invariant irreducible plane curve of degree $d<nq(q-1)$ where $q=n^2$.
Then either $\cC$ is the Hermitian curve $\cH_n$,  or
$n=q^3+1$ and $\cC$ has homogeneous equation
$$ Y^{n^3}Z+YZ^{n^3}-X^{n^3+1}  -\lambda(Y^nZ+YZ^n-X^{n+1} )^{q-n+1}=0, \;   \lambda \in K \backslash \{1\}.$$
For $\lambda=1$, $\cC$ splits into $n^3+1$ lines.
\end{theorem}
\begin{proof} By Lemma \ref{pencil7oct}(i), $\Lambda$ is a $\PGU(3,q)$-fixed pencil. Thus the proof of Lemma \ref{lem4set21} can be adapted to show that $\cC$ belongs to $\Lambda$. This, together with Lemma \ref{pencil7oct}(ii) proves the claim.
\end{proof}
Theorem \ref{thm:pgu} has the following corollary.
\begin{theorem}
\label{the07oct21} Let $q=n^2$. Then $d(\PGU(3,n))=n+1$, $\varepsilon(\PGU(3,n))=qn-1$, and the Hermitian curve $\cH_n$ is the only $\PGU(3,n)$-invariant irreducible plane curve of degree $d=n+1$.
\end{theorem}

\subsection{$NS_q$-invariant curves with $NS_q$ the normalizer of a Singer group}
In this case, our method relying on a fixed pencil is not efficient enough to solve Questions (i),(ii), (iii). Nevertheless, our method becomes appropriate whenever applied to a fixed net.
To show this, we first exhibit an $S_q$-fixed net. As explained in Section \ref{sbp}, up to a change of the reference system over $\mathbb{F}_{q^3}$, a generator of the Singer subgroup $S_q$ is
the projectivity $\sigma: (X,Y,Z)\mapsto (b X,b^{q^2+1}Y,Z)$ for a ($q^2+q+1$)-st primitive root $b\in \mathbb{F}_{q^3}$, while a complement is generated by the projectivity $\mu: (X,Y,Z)\mapsto (Y,X,Z)$.
\begin{lemma}
\label{lem110921} Let $\Gamma$ be the net consisting of all plane curves $\cC_{(\lambda:\mu:\tau)}$ of homogeneous equation
\begin{equation}
\label{eq11set21}
\lambda X^{q+1}Y + \mu Y^{q+1}Z + \tau Z^{q+1}X=0.
\end{equation}
Then $\Gamma$ is $S_q$-fixed, that is, $\cC_{(\lambda:\mu:\tau)}$ is $S_q$-invariant for any non-trivial homogeneous triple $(\lambda: \mu\ :\tau)$.
\end{lemma}
\begin{proof} By $b^{q^2+q+1}=1$, the claim follows from $(bX)^{q+1}(b^{q^2+1}Y)=bX^{q+1}Y, (b^{q^2+1}Y)^{q+1}Z=bY^{q+1}Z$, and $Z^{q+1}(bX)=bZ^{q+1}X$.
\end{proof}
It should be noticed that $\Gamma$ is not $NS_q$-invariant. However, $\cC_{(1:1:1)}$, i.e. the Pellikaan curve, is an $NS_q$-invariant curve in $\Gamma$. There are exactly two more $NS_q$-invariant curves in $\Gamma$, namely $\cC_{(\omega:\omega^2:1)}$, and $\cC_{(\omega^2:\omega:1)}$ where $\omega$ is a primitive third root of unity.
\begin{lemma}
\label{lem110921A} The $S_q$-invariant irreducible plane curves of degree smaller than $2q+1$ belong to the net $\Gamma$.
\end{lemma}
\begin{proof} Let $\cC$ denote an $S_q$-invariant plane curve of degree $d$. Let $P$ be a point of $\cC$ such that the $S_q$-orbit of $P$ has maximal length $q^2+q+1$.
Let $t$ be the tangent line to $\cC$ at $P$, or one of the tangent lines when $P$ is a singular point of $\cC$. Let $r_P=I(P,t\cap \cC)$. Then $r_P\ge 2$, and if $\cC$ is nonclassical then $r_P\ge p$.
There exists a homogeneous triple $(\lambda:\mu:\tau)$ such that $\cC_{(\lambda:\tau:\mu)}$ passes through $P$ and that $t$ is the tangent line to $\cC_{(\lambda:\mu:1)}$, or one of them when $P$ is a singular point of $\cC_{(\lambda:\mu:\tau)}$. Then $I(P,\cC\cap \cC_{(\lambda:\mu:\tau)})\ge 2$, and this holds true for any point in the $NS_q$-orbit of $P$. Therefore either $\cC$ belongs to $\Gamma$, or the B\'ezout theorem yields $d(q+1)=\deg(\cC)\deg(\cC_{(\lambda:\mu:\tau)})\ge 2(q^2+q+1)$ whence $d>2q$.
\end{proof}
A straightforward computation shows that $\cC_{(\lambda:\mu: \tau)}$ has singularity only for three triples, namely $(1:0:0),(0:1:0),(0:0:1)$. In these cases, the curve  splits into lines. Therefore, Lemma \ref{lem110921A} yields that the $NS_q$-invariant irreducible plane curve of degree $q+2$ is the Pellikaan curve together with  $\cC_{(\omega:\omega^2:1)}$, and $\cC_{(\omega^2:\omega:1)}$ where $\omega$ is primitive third root of unity.

Furthermore, since the irreducible plane curve of degree $2q+1$ with homogeneous equation
\begin{equation}
\label{eq13set21} X^{q+1}Y^q + Y^{q+1}Z^q + Z^{q+1}X^q=0
\end{equation}
is $NS_q$-invariant, we have the following.
\begin{theorem}
\label{the130921} $d(NS_q)=q+2$, $\varepsilon(NS_q)=q-1$ and, up to a projectivity defined over $\mathbb{F}_{q^3}$, the only $NS_q$-invariant irreducible curves of degree $d(NS_q)=q+2$ is the Pellikaan curve.
\end{theorem}
\begin{proof} To show the third claim, take $d\in K$ satisfying $d^{q+1}=\omega$, and define $c\in K$ by
$c^qd^{q+2}=1$. Then the projectivity $\pi:(X:Y:Z) \mapsto (cX:dY:Z)$ takes $\cC_{(\omega:\omega^2:1)}$ to $\cC_{(1:1:1)}$ which is the Pellikaan curve. A similar computation shows that $\cC_{(\omega^2:\omega:1)}$ is projectively equivalent to the Pellikaan curve.
\end{proof}
\subsection{$\Delta_q$-invariant curves with $\Delta_q$, the group preserving a triangle in $\PG(2,q)$}
\label{tri}
For the invariant triangle, we choose the one whose vertices are the fundamental points of $\PG(2,q)$, namely $A_1=(1:0:0),A_2=(0:1:0),A_3=(0:0:1)$. Then the Fermat curve $\cF_{q-1}$ of homogeneous equation $X^{q-1}+Y^{q-1}+Z^{q-1}=0$ is $\Delta_q$-invariant. A $\Delta_q$-fixed pencil $\Lambda$ is generated by the irreducible plane curve of homogeneous equation $(XY)^{q-1}+(YZ)^{q-1}+(ZX)^{q-1}=0$ together the reducible plane curve $\cF_{q-1}^{(2)}$ which the square of $\cF_{q-1}$, that is, the plane curve of homogeneous equation $(X^{q-1}+Y^{q-1}+Z^{q-1})^2=0$.  
  Let $\cD$ be a $\Delta_q$-invariant irreducible plane curve whose degree is smaller than $3(q-1)$. Take a point $P\in \cD$ from a long $\Delta_q$-orbit that is not a base point of $\Lambda$. Let $\cC_\lambda$ be the unique curve in $\Lambda$ through $P$. If $\cD$ was not an (irreducible) component of $\cC_\lambda$ then the B\'ezout theorem would yield $\deg(\cD)\ge |\Delta_q|/\deg(\cC_\lambda)=6(q-1)^2/2(q-1)\ge 3(q-1)$. Therefore, $\cD$ is a component of a curve in $\Lambda$. On the other hand, a straightforward computation shows that any curve in $\Lambda$ with homogeneous equation $\lambda (X^{q-1}+Y^{q-1}+Z^{q-1})^2)+
(XY)^{q-1}+(YZ)^{q-1}+(ZX)^{q-1})=0$ is nonsingular. Therefore $\cD$ is a component of $\cF_{q-1}^{(2)}$, and  hence $\cD=\cF_{q-1}$.  This proves the following theorem.
\begin{theorem}
\label{the080921C} $d(\Delta_q)=q-1$, $\varepsilon(\Delta_q)=q-1$, and the Fermat curve of homogeneous equation $X^{q-1}+Y^{q-1}+Z^{q-1}=0$ is the only $\Delta_q$-invariant curve of degree $q-1$. Moreover,
the $\Delta_q$-invariant irreducible plane curves of degree $2(q-1)$ are those of homogeneous equation
$$\lambda(X^{q-1}+Y^{q-1}+Z^{q-1})^2+(XY)^{q-1}+(YZ)^{q-1}+(ZX)^{q-1})=0,\quad \lambda\in K.$$
 \end{theorem}

\subsection{$\PGL(2,q)$-invariant curves}\label{sec:pgl2}
In this section, $p$ is odd. Since $\PGL(2,q)$ preserves an irreducible conic $\cC^2$ of $\PG(2,q)$, we have $d(\PGL(2,q))=2$.

To compute $\varepsilon(\PGL(2,q))$, fix a reference system in $\PG(2,q)$ such that the irreducible conic $\cC^2$ preserved by $\PGL(2,q)$ has homogeneous equation $Y^2-XZ=0$. Then $\PGL(2,q)$ is generated by the projectivities $\tau: (X,Y,Z)\mapsto (Z,Y,X)$, $\sigma_a:(X,Y,Z)\mapsto (X+aY+\ha a^2Z,Y+aZ,Z)$ with $a\in \mathbb{F}_q$, and $\delta_b:(X:Y:Z)\mapsto (b^2X:bY:Z)$. Let $\Sigma=\{\sigma_a|\in \fq\}$, and $\Delta=\{\delta_b|b\in \fq^*\}$.
A straightforward computation shows that both polynomials $Y^{q+1}-(X^q+X)$ and $Y^2-2XZ$ are left invariant by $\tau$, $\Sigma$ and $\Delta$. Let $\Lambda$ be the pencil generated by $\cH_q$ and the conic $\cC^2$ of equation $Y^2-2XZ$ taken $\ha (q+1)$ times.
\begin{lemma}
\label{lem111021B}  The pencil $\Lambda$ consists of all curves $\cC_\lambda$ of homogeneous equation (\ref{eq0}).
Furthermore,
\begin{enumerate}
\item[\rm(i)] $\Lambda$ is a $\PGL(2,q)$-fixed pencil.
\item[\rm(ii)] $\cC_\lambda$ is nonsingular if $\lambda\neq \{1,-1\}$.
\item [\rm(iii)] For $\lambda\not\in \{0,1,-1\}$, a line $t$ of $\PG(2,K)$ meets $\cC_\lambda$ in a unique point if and only if $t$ is a tangent to the conic $\cC^2$ at a point in $\PG(2,q)$.
\end{enumerate}
\end{lemma}
 \begin{proof}(ii) Since $\cC_0=\cH_q$, the claim is true for $\lambda=0$. Assume $\lambda\not\in \{0,1,-1\}$. Let $$V(X,Y,Z)=Y^{q+1}-(X^qZ+XZ^q)-\lambda (Y^2-2XZ)^{(q+1)/2}.$$ If $P=(x:y:z)$ is a singular point of $\cC_\lambda$, then the partial derivatives $V_X,V_Y,V_Z$ vanish at $(x,y,z)$. Let $W(X,Y,Z)=(Y^2-2XZ)^{(q-3)/2}$. Then $V_X=-Z^q+\lambda WZ,V_Y=Y^q-\lambda WY, V_Z=-X^q+\lambda WX$. If $z\neq 0$ then $z=1$ can be assumed, and $V_X(x:y:1)=0$ yields $\lambda W(x:y:1)=1$. Now, $V_Y(x:y:1)=V_Z(x:y:1)=0$ implies $x^q=x,y^q=y$. Hence $y^2-2xz\in \fq$. Thus, $W(x:y:1)\in \{1,-1\}$, a contradiction with $\lambda W(x:y:1)=1$. The same argument works
 if either $x\neq 0$, or $y\neq 0$.

 (iii) The claim holds for the line of equation $Z=0$ which is the tangent to $\cC^2$ at the point $X_\infty=(1:0:0)$. Since $\PGL(2,q)$ acts transitively on the points of $\cC^2$ lying in $\PG(2,q)$, the claim holds true for each tangent $t$ to $\cC^2$ with tangency point in $\PG(2,q)$. To show that no other line has this property, write the equation of $t$ in the form $X=mY+b$, or $Y=b$ with $m,b\in K$. In the latter case, if the property were true the polynomial $X^q+X-\lambda(b^2-2X)+b^{q+1}$ would not have any root in $K$. In the former case, $t$ has a unique common point with $\cC_\lambda$ if and only  $$g(Y)=Y^{q+1}-((mY+b)^q+mY+b)^q-\lambda(Y^2-2(Ym+b))^{(q+1)/2}=\alpha(Y-\beta)^{q+1}$$ for some $\alpha,\beta\in K$.
 Since $g(Y)/dY=\alpha(Y-\beta)^q=\alpha Y^q-\alpha\beta^q$, and the coefficient of $Y^{q-2}$ in $q(Y)/dY$ must vanish, this yields $\lambda(m^2+2b)=0$ whence $m^2=-2b$. Therefore, $Y^2-2(mY+b)=(Y-m)^2$, and hence
 $g(Y)=(1-\lambda)Y^{q+1}-(m^q-\lambda m)Y^q-(m-\lambda m^q)Y+\ha m^{2q}+\ha m^2-\lambda m^{q+1}=0$. Thus, $g(Y)/dY= (1-\lambda)Y^q-(m-\lambda m^q)$, and hence $\beta=(m-\lambda m^q)/(1-\lambda)$.
 Also, $(\lambda-1)\beta=m^q-\lambda m$ from $g(Y)=\alpha(Y-\beta)^{q+1}$. Since $\lambda\neq 1$, the latter two equations yield $-(m-\lambda m^q)=\lambda m-m^q$. As $\lambda\neq -1$, $m^q=m$ follows. Thus $m$ and $b=-\ha m^2$
 are elements of $\fq$, that is, the line $t$ is tangent to $\cC^2$ at a point in $\PG(2,q)$. \end{proof}


  The two exceptions are described in the following proposition.
 \begin{proposition}
 \label{pro090921} Both $\cC_{1}$ and $\cC_{-1}$ are singular curves. Furthermore,
  \begin{enumerate}
 \item [\rm(i)] $\cC_{1}$ splits into the  $q+1$ lines of $\PG(2,q)$ which are the common tangents to $\cC^2$ and to $\cH_q$.
 \item[\rm(ii)] $\cC_{-1}$ is a rational curve whose singular points are exactly the internal points of $\cC$ in $\PG(2,q)$ and each of them has multiplicity $2$.
 \end{enumerate}
 \end{proposition}
 \begin{proof} To show (i), we observe that $\cC^2$ and $\cH_q$ have the same tangent lines at each of their common points. Actually, they are exactly their points in $\PG(2,q)$. Since both $\cH_q$ and $\cC^2$ are $\PGL(2,q)$-invariant, the product of these tangents is a plane curve $\cC$ of degree $q+1$ which is also $\PGL(2,q)$-invariant. Since one of the components of $\cC$ is the line of equation $Z=0$,  the term $Y^{q+1}$ in the homogeneous equation of $\cC$ is missing. From $\deg(\cC)=q+1<q^2-q$, $\Lambda$ contains $\cC$ and hence $\cC$ has homogeneous equation (\ref{eq0}). This yields $\cC=\cC_1$.

 To show (ii), it is useful to recall a couple of known facts on conics
 in $\PG(2,q)$, namely $\cC^2$ has exactly $\ha q(q-1)$ internal points, i.e. points not covered by tangents to $\cC^2$ at points in $\PG(2,q)$, and the same number of external lines, i.e. lines disjoint from $\cC^2$ in $\PG(2,q)$. Furthermore, the stabilizer of any such external line $\ell$ is a dihedral subgroup $D_{q+1}$  of $\PGL(2,q)$ of order $2(q+1)$ where $D_{q+1}$ acts on $\ell\cap \PG(2,q)$ with two orbits, namely  $I(\cC^2)\cap \PG(2,q)$ and its complementary set both of length $\ha(q+1)$.  On $\ell\cap\PG(2,K)\setminus \PG(2,q)$, $D_{q+1}$ has an orbit of length $2$ (namely $\cC\cap \ell$)  while the other orbits have length $q+1$.
 Now take a point $P\in I(\cC^2)$.  Since the base points of $\Lambda$ are exactly the $q+1$ common points of $\cC^2$ and $\cH_q$ (each counted with multiplicity  $2$ $q+1$), there is a unique curve $\cC_\lambda\in \Lambda$ through $P$.  By (i) each external point to $\cC^2$ is contained in $\cC_1$. Therefore, no external point of $\cC^2$ is contained in $\cC_\lambda$. Also, no point of $\cC^2$ in
 $\PG(2,K)\setminus \PG(2,q)$ is contained in $\cC_\lambda$. As $\PGL(2,q)$ acts transitively on $I(\cC^2)$, each internal point of $\cC^2$ is contained in $\cC_\lambda$.  Let $\ell$ be an external line to $\cC^2$ through $P$. Since both $\cC_\lambda$ and $\ell$ are $D_{q+1}$-invariant, $D_{q+1}$ preserves $\cC_\lambda\cap\ell$. From $P\in \cC_\lambda\cap \ell$ and $\deg(\cC_\lambda)=q+1$ it follows  $\cC_\lambda \cap \ell=I(\cC^2)\cap \ell$. In particular, $I(P,\cC_\lambda\cap \ell)=2$, and this holds for each of the $\ha(q+1)$ external lines to $\cC^2$ through $P$. Furthermore, if $t$ is the tangent to $\cC^2$ (and $\cH_q)$ at $P$ then $I(P,\cC_\lambda\cap t)=q+1$ as $I(P,\cH_q\cap t)=q+1$, $I(P,\cC_0\cap t)=q+1$ and $\Lambda$ is generated by $\cH_q$ and $\cC_0$. Therefore, $P$ is a double point of $\cC_\lambda$, and $t$ is the unique tangent to $\cC_\lambda$ at $P$. Thus, $\cC_\lambda$ has at least $\ha q(q-1)$ double points. As $\deg(\cC_\lambda)=q+1$, this implies that $\cC_\lambda$ is a rational curve.  Finally, $\lambda=-1$. In fact, the line $Y=0$ meets $\cC_{-1}$ in the points $Q=(\xi:0:1)$ with $(-\xi)^{(q-1)/2}=-1$ other than the origin $O=(0:0:1)$ and $X_\infty=(1:0:0)$. Apart from the latter points lying on $\cC^2$, $Q$ is an internal point of $\cC^2$.
  \end{proof}
 \begin{theorem}
 \label{th131021} For $\lambda\not\in \{0,1,-1\}$, the full automorphism group $\aut(\cC)$ of the nonsingular plane curve $\cC_\lambda$ with homogeneous equation (\ref{eq0}) is isomorphic to $\PGL(2,q)$.
 \end{theorem}
 \begin{proof} By Lemma \ref{lem111021B}(i) $\aut(\cC_\lambda)$ contains a subgroup isomorphic to $\PGL(2,q)$.  Since $\cC_\lambda$ is a nonsingular plane curve of degree $>3$, $\aut(\cC)$ is linear,  see \cite[Theorem 11.29]{hirschfeld-korchmaros-torres2008}. From Lemma \ref{lem111021B}(iii), the set consisting of the $q+1$ $\fq$-rational points of the conic $\cC^2$ is invariant under the action of $\aut(\cX)$. For $q=3$, this set is a non-degenerate quadrangle and hence $\aut(\cX)$ is a subgroup of $\rm{Sym}_4$. Since $\rm{Sym}\cong PGL(2,3)$ the claim follows. If $q\ge 5$, an irreducible conic is uniquely determined by any five of its points, and hence $\cC^2$ is itself $\aut(\cX)$-invariant. Therefore, $\aut(\cC_\lambda)$ is isomorphic to a subgroup of $\PGL(2,q)$, and the claim follows.
 \end{proof}
 \begin{proposition}
 \label{pro241021} For  $\lambda\neq -1,\infty$, $\aut(\cC_\lambda)$ has exactly three short orbits on $\cC_\lambda$, two of length $\ha (q^3-q)$ and one consisting of the points of $\cC^2$.
 \end{proposition}
 \begin{proof} From Theorem \ref{th131021}, $\aut(\cC_\lambda)=\PGL(2,q)$. We show that a non-trivial element of $\PGL(2,q)$ has a fixed point outside $\PG(2,q)$ if and only if it is an involution. From Dickson's classification of subgroups of $\PGL(2,q)$, see \cite[Theorem A.8, Remark 11.115]{hirschfeld-korchmaros-torres2008}, $\PGL(2,q)$ has a partition in subgroups with three components up to conjugacy, namely (a) a subgroup of order $q$ whose non-trivial elements fix exactly one point $P$ and $P\in\cC^2$, (b) a cyclic group of order $q\pm 1$ whose non-trivial elements of order larger than two fix exactly one point $P\in \PG(2,q)\setminus \cC^2$. Furthermore, any involution of $\PGL(2,q)$ is a homology with center in $\PG(2,q)$ whose axis $\ell$ is a line of $\mathbb{F}_q$. Therefore, the common points of $\ell$ and $\cC_\lambda$ are fixed points of the involution. Assume that $\ell\cap\cC_\lambda$ only consists of points in $\PG(2,q)$. If a point $P\not\in\cC^2$ is in the intersection then $P$ is not a base point of the pencil $\Lambda$ and hence there exists a unique curve in $\Lambda$ passing through $P$. From Proposition \ref{pro090921}, that curve must be either $\cC_1$, or $\cC_{-1}$, a contradiction. Therefore $\ell\cap \cC_\lambda=
\ell\cap \cC^2=\{Q_1,Q_2\}$. Observe that $\ell$ is not a tangent to $\cC^2$, and hence neither to $\cC_\lambda$ at $Q_i$. Also, $Q_i$ is a nonsingular point. Therefore, $I(Q_i,\ell\cap \cC_\lambda)=1$. Since $\deg(\cC_\lambda)>2$, this contradicts B\'ezout's theorem. Moreover, let $P\in \ell\cap \cC_\lambda$ with $P\not\in\PG(2,q)$. Then the stabilizer of $P$ in $\PGL(2,q)$ has order $2$ and hence the orbit of $P$ under the action of $\PGL(2,q)$ has length $\ha (q^3-q)$. Since $\PGL(2,q)$ has exactly two conjugacy classes of involutions, the claim follows.
 \end{proof}
 From the above proposition and discussion, the following two theorems follow.
\begin{theorem}
\label{the00921} $d(\PGL(2,q))=2$, $\varepsilon(\Delta_q)=q-1$, and the irreducible conic of equation $Y^2-XZ=0$ is the unique $\PGL(2,q)$-invariant irreducible plane curve of degree smaller than $q+1$.
\end{theorem}
 \begin{theorem}\label{thm:pgl2A}
If $\cC$ is a $\PGL(2,q)$-invariant irreducible plane curve of degree $2<d<q^2-q$ then $\cC$ is either the Hermitian curve, or it has homogeneous equation (\ref{eq0}) with $\lambda\in K$.
\end{theorem}

\section{Plane curves with nonlinear automorphism group isomorphic to $\PGL(2,q)$}
\label{hemis} In \cite{KNS2017}, the irreducible plane curve $\cK$ with affine equation (\ref{eq9oct21}) came up as a useful tool in the construction of a new family of hemisystems of the Hermitian surface. In Section   \ref{BKSc} we show some more of its properties, namely $\cK$ is an unramified  degree-2 cover of the Hermitian curve $\cH_q$ and the automorphism group $G$ of $\cK$ is isomorphic to $\PGL(2,q)\times C_2$.  It was observed in the proof of \cite[Lemma 5.3]{KNS2017} that $G$ is generated by the following maps.
\begin{enumerate}
\label{formula}
\item[]$\beta_1:(X,Y) \mapsto (X+\alpha, Y+\alpha), \alpha \in \fq$;
\item[] $\beta_2:(X,Y)\mapsto (\mu X, \mu Y), \mu \in \fq^*$;
\item[] $\beta_3:(X,Y) \mapsto (X^{-1},Y^{-1})$;
\item[] $\beta_4:(X,Y)\mapsto (Y,X)$.
\end{enumerate}
More precisely, $\beta_1,\beta_2,\beta_3$ generate a group $H\cong \PGL(2,q)$ and $\beta_4$ centralizes $H$.
Clearly, $\beta_1,\beta_2,\beta_4$ are in $\PGL(3,q)$ but this does not hold true for $\beta_3$ as the homogeneous equation of $\beta_3$  is $(X,Y,Z)\mapsto (YZ,XZ,XY)$. Therefore, $\cK$ is an irreducible plane curve which is left invariant under a nonlinear $\PGL(2,q)$.
\begin{proposition}
\label{pro101021} The subgroup $G\cap \PGL(3,q)$ has order $2q(q-1)$.
\end{proposition}
\begin{proof} Let $L=\{\beta_1\beta_2|\,\alpha\in \fq,\mu \in \fq^*\}$. Then $L$ is a subgroup of $H$ of order $q(q-1)$. In particular, $L$ is a maximal subgroup of $H$. Therefore, the elements of $H$ other than those in $L$ are not in $\PGL(3,q)$.  Since $\beta_4\in \PGL(3,q)$ and each element of $G$ is a product of $\beta_4$ and element from $H$, the claim follows.
\end{proof}
Let $\Lambda$ be the pencil generated by $\cK$ together with the plane curve $(X-Y)^{q+1}Z^{q-1}=0$.
\begin{lemma}
\label{lem091021} The pencil $\Lambda$ consists of all plane curves $\cC_\lambda$ with affine equation
\begin{equation}
\label{eq00}
(X+Y)^{q+1}-2((XY)^{q}+XY)-\lambda (X-Y)^{q+1}=0, \quad \lambda \in K\cup \{\infty\}.
\end{equation}
 Furthermore,
\begin{itemize}
\item[(i)] $\Lambda$ is a $G$-fixed pencil.
\item[(ii)] $\cC_1$ splits into $2q$ lines, namely the lines of equations $X=\xi Z$ and $Y= \xi Z$ when $\xi$ ranges over $\fq$.
\item[(iii)] For $\lambda\neq 1,-1,\infty$, $\cC_\lambda$ is an irreducible plane curve with $q+2$ singular points. Two of them are ($q-1$)-fold points each with a unique tangent line and the remaining $q$ are nodal double points. The intersection multiplicity of $\cC_\lambda$ with its tangent line at a $(q-1)$-fold point is equal to $q$.
\item[(iv)] $\cC_{-1}$ has two irreducible components, namely the rational curves of equation $YZ^{q-1}-X^q=0$ and $XZ^{q-1}-Y^q=0$, respectively.
\end{itemize}
\end{lemma}
\begin{proof} (i) Set $F(X,Y,Z)=(X+Y)^{q+1}Z^{q-1}-2((XY)^{q}+XYZ^{2q-2}),$ $G(X,Y,Z)=(X-Y)^{q+1}Z^{q-1}$, and $V(X,Y,Z)=F(X,Y,Z)-\lambda G(X,Y,Z)$. Then $\cC_\lambda$ has homogeneous equation $V(X,Y,Z)=0$. Moreover,
$V(X,Y,Z)=V(X+\alpha,Y+\alpha,Z)=V(\mu X,\mu Y,Z)=V(Y,X,Z)$, for all $\alpha\in \fq,\mu\in\fq^*$, and $V(YZ,XZ,XY)=(XY)^{q-1}Z^2V(X,Y,Z)$. Thus, $\Lambda$ is fixed by a generator set of $G$.

(ii) For $\lambda=1$, a straightforward computation shows $V(X,Y,Z)=-2(X^q-X)(Y^q-Y)$ whence the claim follows.

(iii) If a  point $P=(x:y:z)$ is a singular point of $\cC_\lambda$, then the partial derivatives $V_X,V_Y,V_Z$ vanish in $(x,y,z)$. A straightforward computation shows that $V_X=V_Y$ implies $(Y-X)Z^{2q-2}=0$. Assume first $z\neq 0$. Then $x=y$. Let $P_\xi=(\xi:\xi:1)$ for $\xi\in \fq$. Then $P_\xi$ is a double point with tangent lines of equation $X=\xi Z$ and $Y=\xi Z$, respectively. Since $\deg(\cC_\lambda)=2q$, B\'ezout's theorem applied to $\cC_\lambda$ and the line $\ell$ of equation $X-Y=0$ implies that the points $P_\xi$ with $\xi\in\fq$ are the only singular points of $\cC_\lambda$ on $\ell$, unless $\ell$ is a component of $\cC_\lambda$. This exception only occurs for $\lambda=\infty$. Assume now $z=0$. For $\lambda\neq \infty$, the points of $\cC_\lambda$ on the infinite line $\ell_\infty$ of equation $Z=0$ $P(x:y:0)$ are two, namely $X_\infty=(1:0:0)$ and $Y_\infty=(0:1:0)$, both of multiplicity $q-1$. For $\lambda \in K\setminus\{1\}$, the infinity line is the unique tangent to $\cC_\lambda$ at $X_\infty$ (as well as at $Y_\infty$). It remains to show that $\cC_\lambda$ is reducible only for $\lambda\in \{1,-1,\infty\}$. Let $\lambda\neq 1,\infty$. Assume that $\cC_\lambda$ has an irreducible component $\cU$ that does not pass through the origin $O=(0:0:1)$. If $\ell_X$ is the line of equation $Y=0$ then $\ell_X\cap \cC_\lambda=\{O,X_\infty\}$, and hence $\ell_X \cap \cU=\{X_\infty\}$. Therefore, $I(X_\infty,\ell_X\cap \cU)=\deg(\cU)$. Since $\ell_X$ is not a tangent to $\cC_\lambda$ at $X_\infty$ the same holds for $\cU$. Thus, $X_\infty$ is a $\deg(\cU)$-fold point of
$\cF$ and hence $\cU$ is the infinity line $\ell_\infty$. But then $\lambda=\infty$. Therefore each irreducible component of $\cC_\lambda$ passes through $O$. Since $O$ is a double point of $\cC$ whereas $\cC$ has at least two irreducible components,  $\cC_\lambda$ turns out to have exactly two irreducible components, say $\cD$ and $\cE$, so that $O$ is a simple point for both with tangent lines $\ell_X$ and $\ell_Y$, respectively. Since $\beta_1$ preserves $\cC_\lambda$ and $\beta_1$ has order $p\neq 2$, both $\cD$ and $\cE$ are $\beta_1$-invariant. The subgroup $M$ of $H$ consisting of all maps $\beta_1$ has order $q$ and acts transitively on the set $D$ of the points $P_\xi(\xi:\xi:1)$ with $\xi\in \fq$. Therefore, $\deg(\cD)\geq q$ and $\deg(\cE)\geq q$ whence $\deg(\cD)=\deg(\cE)=q$ follows. Moreover, $I(O,\ell_X \cap \cD)=q$ as $q+1=I(O,\ell_X\cap \cC)=I(O,\ell_X \cap \cD)+I(O,\ell_X \cap \cD)=I(O,\ell_X \cap \cD)+1$. This shows that $X_\infty\not\in \cD$. The same holds for $\cE$ and the point $Y_\infty$. Therefore, $Y_\infty$ is a ($q-1$)-fold point of $\cD$.  This shows that $\cD$ has affine equation $Y=f(X)=c X^q+\ldots+dx=0$ with $c,d\in \fq$. Since $\cD$ is $M$-invariant, the polynomial $f(X)$ is an $\fq$-linearized polynomial, that is, $f(X)=cX^q+dX$. Actually, $d=0$ as the tangent to $\cD$ at $O$ is the line $Y_\infty$. Since $\beta_4$ takes $\cD$ to $\cE$, an affine equation for $\cE$ is $X+cY^q=0$. It turns out that $\cC_\lambda$ has affine equation $u(Y+cX^q)(X+cY^q)=0$ for some $u\in K$. Since $\cC_\lambda$ is a curve in the pencil $\Lambda$, this only occurs when $\lambda=-1$, $c=-1$, $u=-2$.

(iv) For $\lambda=-1$, the final argument in the proof of (iii) shows that $\cD$ and $\cE$ are the only irreducible components of $\cC_1$.
\end{proof}

\label{noncl}
\section{nonclassicality of $G$-invariant irreducible plane curves}\label{noncl}
Within the scope of the St\"ohr-Voloch theory, a  fundamental question concerns   the nonclassicality or Frobenius nonclassicality of a given curve.
Curves possessing such features are somewhat rare, but their characterization is relevant due to applications in  many   areas, such as coding theory and finite geometry.
The Frobenius nonclassical  property  in some of our  G-invariant curves have  already  been   established in the literature. For instance, it is well know that the Hermitian curve over
 $\mathbb{F}_{q^2}$ is Frobenius nonclassical and that the DGZ curve is an example of a double-Frobenius nonclassical curve.

In this section we present new Frobenius nonclassical curves  by characterizing all such  curves among those  $\mathcal{C} \subseteq \PG(2,K)$ of degree $d<q^3(q^2-1)$
that are  $\PGU(3,q)$-invariant. More precisely, we prove the following.

\begin{theorem}
 Let $\mathcal{C} \subseteq \PG(2,K)$ be the $\mathrm{PGU}(3, q)$-invariant  curve
$$
Y^{q^{3}} Z+Y Z^{q^{3}}-X^{q^{3}+1}-\lambda\left(Y^{q} Z+Y Z^{q}-X^{q+1}\right)^{q^{2}-q+1}=0,\quad  \lambda \in K \backslash\{1\}.
$$
Then $\mathcal{C}$ is  $\mathbb{F}_{q'}$-Frobenius nonclassical  if and only if  either   $q'=q^6$ and   $\lambda=0$, or
$q^{\prime}=q^4$ and  $\lambda^{q+1}=1$.
In the latter case, the curve $\mathcal{C}$  has exactly  $\left(q^{3}+1\right)\left(q^{4}-q^{3}+1\right)$ rational points over $\mathbb{F}_{q^4}$.
\end{theorem}

\begin{proof}
Let us set  $F(X,Y,Z)=Y^{q}Z+YZ^q-X^{q+1}$,  $G(X,Y,Z)=Y^{q^3}Z+YZ^{q^3}-X^{q^{3}+1}$, and  $V(X,Y,Z)=G(X,Y,Z)- \lambda F(X,Y,Z)^{q^2-q+1}$.
Considering the partial derivatives of $V$,  the polynomials $U_1,U_2$, and $U_3$,  defined as follows
\begin{equation}\label{partial-deriv}
\begin{cases}
V_X=-X^{q^{3}}+\lambda F^{q^{2}-q}X^{q}=U_1^q\\
V_Y=Z^{q^{3}}-\lambda F^{q^{2}-q}Z^{q}=U_2^q\\
V_Z=Y^{q^{3}}-\lambda F ^{q^{2}-q}Y^{q}=U_3^q,
\end{cases}
\end{equation}
are such that  $U_1^qX+U_2^qY+U_3^qZ=V$. Thus,  from \eqref{eqnoc} and \eqref{eqnoc1}  in Section 2, the curve  $\mathcal{C}$  is nonclassical, and it is
$\mathbb{F}_{q'}$-Frobenius nonclassical if and only if
\begin{equation}\label{Frob-condition1}
U_1X^{q'/q}+U_2Y^{q'/q}+U_3Z^{q'/q} \equiv 0  \pmod V.
\end{equation}
Note that if $\lambda=0$, then $\mathcal{C}$ is the curve $Y^{q^{3}} Z+Y Z^{q^{3}}-X^{q^{3}+1}=0$, which is clearly   $\mathbb{F}_{q^6}$-Frobenius nonclassical.
Let us assume  that $\lambda\neq 0$  and $\mathcal{C}$  is $\mathbb{F}_{q'}$-Frobenius nonclassical.  Then \cite[Theorem 8.77]{hirschfeld-korchmaros-torres2008} gives $\sqrt{q^{\prime}}+1 \leq q^3+1 \leq \frac{q^{\prime}-1}{q-1}$, that is,  $q^4-q^3+q\leq q^{\prime}\leq q^6$. It is clear that   $q^{\prime}>q^3$, and we claim that $q^{\prime}<q^6$.  Indeed, since  $\mathcal{C}$ is nonsingular of degree $d=q^3+1$,  if
$q^{\prime}=q^6$, then $d=\sqrt{q'}+1$ and   \cite[Corollary 3.2]{bor-hom}  imply   that $\mathcal{C}$  is projectively equivalent to  the curve $X^{q^3+1}+Y^{q^3+1}+Z^{q^3+1}=0$.  This is  clearly impossible, as the partial derivatives of  $V(X,Y,Z)$ in \eqref{partial-deriv} are not $q^3$ powers of linear forms. Hence
$q^3<q'<q^6$. In addition,  note that \eqref{partial-deriv} also  yields

\begin{equation}\label{Frob-condition2}
U_1X^{q'/q}+U_2Y^{q'/q}+U_3Z^{q'/q} \equiv (YZ^{q'/q^3}+ZY^{q'/q^3}  -X^{q'/q^3+1})^{q^2} \pmod F.
\end{equation}
Thus, since  the $\mathbb{F}_{q^2}$-rational points of $F=Y^{q}Z+YZ^q-X^{q+1}=0$  are points of the curve $\mathcal{C}$, it follows from  \eqref{Frob-condition1} and   \eqref{Frob-condition2} that all such $q^3+1$ points lie on the curve
$$YZ^{q'/q^3}+ZY^{q'/q^3}  -X^{q'/q^3+1}=0,$$
and then $q^3<q'<q^6$ implies   $q'=q^4$. Now using   \eqref{partial-deriv} and   \eqref{Frob-condition1},  a  direct  computation gives $F^{q^3}(1-\lambda^{q+1})  \equiv 0  \pmod V$, and then
$\lambda^{q+1}=1$.  The converse is straightforward. The number of $\mathbb{F}_{q^4}$-rational points of $\mathcal{C}$ follows directly from the Hefez-Voloch formula for nonsingular Frobenius nonclassical curves, see \cite[Theorem 1]{hefez-voloch}.
\end{proof}

\section{On the number of $\mathbb{F}_{q^i}$-rational points of the DGZ curve}
\label{ratpointdgz}
In this Section, we show how knowledge of the (linear) action of a subgroup of $\PGL(3,q)$ preserving a plane curve can be used to count its points in $\PG(2,q^i)$ for $i\ge 1$. For this purpose,
we refer to the DGZ curve $\cF_{3,1}$ whose number $N_i$ of $\mathbb{F}_{q^i}$-rational points has already been computed in \cite{ bor,DGZ} for $i=1,2,3$. More precisely, $N_1=0,N_2= q^4-q$, and $N_3 = q^6-q^5-q^4+q^3$.
Our result is stated in the following theorem.
\begin{theorem}\label{thm:pointsdgz}
$$
N_i =
\begin{cases}
q^4-q,       & i = 4,\\
0,              & i = 5,\\
q^6-q^5+q^3-q, &i =6.
\end{cases}
$$
In particular, the DGZ curve has no $\mathbb{F}_{q^4}$-rational point other than its $\fqq$-rational points, and no $\mathbb{F}_{q^6}$-rational point other than its $\mathbb{F}_{q^3}$ and $\fqq$-rational points.

\end{theorem}
\begin{proof}
Since $|\aut(\cF_{3,1})| = |\PGL(3,q)| > 84(\mathfrak{g}(\cF_{3,1})-1)$, by \cite[Theorem 11.56]{hirschfeld-korchmaros-torres2008}, $\PGL(3,q)$ has at most three short orbits on $\cF_{3,1}$. In 	 \cite{DGZ}, it is actually shown that there are exactly two  $\PGL(3,q)$-short orbits on $\cF_{3,1}$, one tame, which coincides with $\cF_{3,1}(\mathbb{F}_{q^3})$ and one non-tame, which coincides with $\cF_{3,1}(\fqq)$.

We begin with the $\mathbb{F}_{q^4}$-rational points.  Let $P \in \cF_{3,1}$ be an $\mathbb{F}_{q^4}$-rational point of $\cF_{3,1}$ which is not $\fqq$-rational. Then the orbit $\Omega$ of $P$ under $\PGL(3,q)$ is long, that is to say,
\begin{equation}\label{hw1}
|\cF_{3,1}(\mathbb{F}_{q^4})| \geq q^4-q + q^3(q^2+q+1)(q-1)^2(q+1).
\end{equation}

From the Hasse-Weil Bound,

\begin{equation}\label{hw2}
|\cF_{3,1}(\mathbb{F}_{q^4})| \leq q^4+1+2(\ha q(q-1)(q^3-2q-2)+1)q^2.
\end{equation}

By comparing Inequalities \eqref{hw1} and \eqref{hw2}, a computation shows

$$
q^8 - q^6 - q^5 + q^4 + q^3 - q<q^7 - q^6 - 2 q^5 + q^4 + 2 q^3 + 2 q^2 + 1,
$$

which holds true only for $ -1 < q < 1.38552$, a contradiction.

Arguing in a similar way as in the case $i =4$, we see that $\cF_{3,1}(\mathbb{F}_{q^5})$ is a union of $\PGL(3,q)$-long orbits, hence

\begin{equation}\label{hw51}
|\cF_{3,1}(\mathbb{F}_{q^5})| \geq  q^3(q^2+q+1)(q-1)^2(q+1).
\end{equation}

This time, the Hasse-Weil Bound reads

\begin{equation}\label{hw54}
|\cF_{3,1}(\mathbb{F}_{q^5})| \leq q^5+1+2(\ha q(q-1)(q^3-2q-2)+1)q^{5/2}.
\end{equation}

Since the lower bound in  \eqref{hw51} is smaller than the upper bound \eqref{hw54} (assuming $q$ positive) only for $q < 1,53095$, we get a contradiction.

The proof for the case $i =6$ is analogous to the case $i=4$.
\end{proof}
\section{Some more results on the curves in Sections \ref{sec:pgl2} and \ref{hemis}}
\label{BKSc}
\subsection{Genus and $p$-rank of the irreducible plane curve of homogeneous equation (\ref{eq0})} In this section, $\lambda\not\in \{0,1,-1\}$, and $\cY=\cC_\lambda$ stands for the nonsingular plane curve with homogeneous equation (\ref{eq0}), and $\aut(\cY)$ for the automorphism group of $\cY$. Clearly, the genus $\mathfrak{g}(\cC_\lambda)$ is equal to $\ha(q^2-q)$.

To compute the $p$-rank $\gamma$  of $\cC_\lambda$ by means of the Deuring-Shafarevic formula, we first determine the quotient curve $\cC_\lambda/\Sigma$ with $\Sigma$ being defined in Section \ref{sec:pgl2}. Let $F_1$ be the function field of $\cC_\lambda$, that is, $F=F(x,y)$ with $y^{q+1}-(x^{q}+x)- \lambda (y^{2}-2x)^{(q+1)/2}=0$. For $\eta=y^q-y$ and $\xi=y^2-2x$, let $F_0$ be the subfield of $F_1$ generated by $\xi$ and $\eta$. Then
$\deg(F_1|F_0)=q$. Also, $\sigma_a(\eta)=(y+a)^q-(y+a)=y^q-y=\sigma_a(\eta)$ and $\sigma_a(\xi)=(y+a)^2-2(x+ay+\ha a^2)=y^2-2x=\xi$. Thus, $F_0$ is the fixed field of $\Sigma$. Moreover,
$$\eta^2-\xi^q-\xi-2\lambda\, \xi^{(q+1)/2}=(y^q-y)^2-y^{2q}+2x^q-y^2-2x-2\lambda (y^2-2x)^{(q+1)/2}=0.$$
Since $V^2-U^q-U-2\lambda\, U^{(q+1)/2}$ is an irreducible polynomial in $K[U,V]$, it follows that $F_0=F(\xi,\eta)$ with $\eta^2=\xi^q+\xi+2\lambda\, \xi^{(q+1)/2}$.
 As $F_0$ has genus $\bar{\mathfrak{g}}=\ha(q-1)$, the latter equation can also be written as $\eta^2=\xi^{2\bar{\mathfrak{g}}+1}+\tau \xi^{\bar{\mathfrak{g}}+1}\xi+\xi$ where $\tau=2\lambda$.
The Hasse-Witt matrix of hyperelliptic function fields with such a special equation was computed by Miller who also showed its invertibility when $p\nmid \bar{\mathfrak{g}}$ apart from a finitely many values of $\tau$,  see \cite[Proposition 1]{Mi}. Miller's result applies to $F_0$ as in our case $\bar{\mathfrak{g}}=\ha(q-1)$ and $q$ is a power of $p$. Thus,  $F_0$ is an ordinary function field, apart from finitely many $\lambda \in K$.
Discarding these exceptions, the $p$-rank $\bar{\gamma}$ of $F_0$ is equal to $\mathfrak{g}(F_0)=\ha (q-1)$.

Since each non-trivial element of $\Delta$ fixes a unique point of $\cC_\lambda$, namely $X_\infty=(1:0:0)$, the Deuring-Shafarevic formula applied to $\Sigma$ reads $\gamma-1=q(\bar{\gamma}-1)-(q-1)$ whence $\gamma=\ha (q^2-q)=\mathfrak{g}(\cC_\lambda)$. Therefore, the following result is obtained.
\begin{theorem}
\label{th131021B} Apart from a finitely many $\lambda\in K$, the curve $\cC_\lambda$ is an ordinary curve.
\end{theorem}
\begin{rem}
\label{rem141021} {\emph{For any power $q$ of $p>2$, Theorem \ref{th131021B} implies the existence of infinitely many ordinary curves $\cX$ with $\aut(\cX)\cong \PGL(2,q)$  such that $|\aut(\cX)|>8 \mathfrak{g}(\cX)^{3/2}$. It should be noted that, for solvable groups  of automorphisms of ordinary curves, the best known bound  on the order of the group is $ 34(\mathfrak{g}-1)^{3/2}$,  see \cite{KM, MS1}. It is still an open problem whether  this bound holds true also for non-solvable groups, apart from a few exceptions.  }}
\end{rem}

\subsection{Genus and the automorphism group of the irreducible plane curve of homogeneous equation (\ref{eq00})}

In this section, $\cX=\cX_\lambda$ stands for a nonsingular model of the irreducible plane curve $\cC_\lambda$ with $\lambda\in K\setminus\{1,-1\}$ as introduced in Section \ref{hemis}.
Let $F=F_\lambda$ be the function field of $\cX_\lambda$, that is, $F=K(x,y)$ with $(x+y)^{q+1}-2((xy)^{q}+xy)-\lambda (x-y)^{q+1}$. Let $\aut(F)=\aut(\cX)$ be the automorphism group of $F$ fixing $K$ element-wise. In terms of $F$, (\ref{formula}) states that the following maps of $F$ are elements of $\aut(\cX)$:
\begin{itemize}
\item[(i)]  $T_\alpha:(x,y)\rightarrow (x+\alpha,y+\alpha),\, \alpha\in \mathbb{F}_q$;
\item[(ii)] $E_\mu: (x,y)\rightarrow (\mu x,\mu y), \, \mu\in \mathbb{F}_q^*$;
\item[(iii)]$L_0: (x,y) \rightarrow (x^{-1},y^{-1})$;
\item[(iv)] $\Phi: (x,y)\rightarrow (y,x)$.
\end{itemize}
Also, the maps $T_\alpha, E_\mu$ together with $L_0$ generate a subgroup $H$ of $\aut(F)$ isomorphic to $\PGL(2,q)$, and $\Phi$ centralizes $H$, so that $L=H\times \langle \Phi\rangle$ is also a subgroup of $\aut(\cX)$.

We begin by showing a close connection between $\cC_\lambda$ and the plane curve of homogeneous equation (\ref{eq0}) with the same $\lambda$. To avoid confusion, the latter curve will be denoted by $\cD_\lambda$.
\begin{proposition}
\label{pro111021} For $\lambda\in \mathbb{K}\setminus \{1,-1\}$, the curve $\cD_\lambda$  is a nonsingular model of the quotient curve $\bar{\cX}=\cX/\langle\Phi\rangle$.
\end{proposition}
\begin{proof} By Lemma \ref{lem111021B}(ii), $\cD_\lambda$ is a nonsingular plane curve. In the function field $F$, let $u=x+y$ and $v=-2xy$. Then
$$u^{q+1}-(v^q+v)-\lambda(u^2-2v)^{(q+1)/2}=(x+y)^{q+1}-2((xy)^{q}+xy)-\lambda (x-y)^{q+1}=0.$$
Therefore, $F_1=K(u,v)$ is the a function field of the irreducible plane curve of homogeneous equation (\ref{eq0}),  see Sections \ref{sec:pgl2} and \ref{BKSc}. Clearly, $[F:F_1]\le 2$. On the other hand, $\Phi(u)=\Phi(x+y)=\Phi(x)+\Phi(y)=y+x=x+y$ and $\Phi(v)=\Phi(-2xy)=-2\Phi(x)\Phi(y)=-2yx=-2xy$. Hence the fixed field of $\Phi$ contains $F_1$. Since $\Phi$ has order $2$, this yields $[F:F_1]=2$.
\end{proof}
\begin{proposition}
\label{pro111021A} The extension $F|F_1$ is an unramified Kummer extension.
\end{proposition}
\begin{proof} $F|F_1$ is a Kummer extension as $\deg(F|F_1)=[F:F_1]=2$ and $p\neq 2$. Assume on the contrary that $F|F_1$ ramifies at a point $\bar{P}$ of $\cD_\lambda$, and denote by $\gamma$ the branch of $\cC_\lambda$ lying over $\bar{P}$. Let $P=(\xi:\eta:\zeta)$ be the center of $\gamma$. Assume first $\zeta \neq 0$. Then $\zeta=1$ can be assumed. Since $\Phi$ fixes $\gamma$ and $\Phi$ is linear, the point $P$ is fixed by $\Phi$ and hence $P=(\xi:\xi:1)$. Actually, there is another branch $\delta$ of $\cC_\lambda$ centered at $P$, and the tangents to $\gamma$ and $\delta$ are the lines of equations $X=\xi Z$ and $Y=\xi Z$, respectively. These two lines are interchanged by $\Phi$ which yields that $\gamma$ and $\delta$ are also interchanged by $\Phi$. In particular, $\Phi$, does not fix $\gamma$, a contradiction. If $\zeta=0$ then $P$ is either $X_\infty=(1:0:0)$ or $Y_\infty=(0:1:0)$. As $\Phi$ interchanges $X_\infty$ with $Y_\infty$, $\Phi$ cannot fix a branch centered at $X_\infty$, or $Y_\infty$.
\end{proof}
Since $\mathfrak{g}(F_1)=\ha (q^2-q)$, the Hurwitz genus formula applied to $\Phi$ gives the following result.
\begin{proposition}
\label{genusBKSc} The genus of $\cX$ is equal to $q^2-q-1$.
\end{proposition}
Furthermore, Proposition  \ref{pro111021A} implies that each $T_\alpha$ has exactly two fixed points on $\cX$, namely those lying over $X_\infty$ in the cover $\cX \mapsto \cX/\langle \Phi \rangle$.

To prove that $\aut(\cX)=H\times \langle \Phi \rangle$, a careful analysis of the action of $H$ on $\cX$ is needed. For this purpose, let $P_1,P_2\in \cX$ be the points arising from the branches of $\cC_\lambda$ centered at the points $X_\infty$ and $Y_\infty$, respectively. Also, let $N_1,N_2\in\cX$ be the points arising from the branches of $\cC_\lambda$ centered at $O=(0:0:1)$.

The group $T=\{T_\alpha|\alpha\in \mathbb{F}_q\}$ is a Sylow $p$-subgroup of $H$. Each nontrivial element in $T$ has exactly two fixed points, namely $P_1$ and $P_2$.  Furthermore, $E=\{E_\mu| \mu\in \mathbb{F}_q^*\}$ is a (cyclic) subgroup of $\aut(\cX)$ of order $q-1$ whose fixed points are $P_1,P_2,N_1,N_2$. The group generated by $E$ together with $L_0$ is a dihedral group of order $2(q-1)$ and $L_0$ interchanges $P_1$ with one of the points $N_1,N_2$, say $N_1$, and it interchanges $P_2$ with $N_2$. The group generated by $L_0$ and $\Phi$ is an elementary abelian group of order $4$ and acts on $\{P_1,P_2,N_1,N_2\}$ as a sharply transitive permutation group.

Here, the normalizer $N_H(T)$ of $T$ in $\aut(\cX)$ is $T\rtimes E$; also the normalizer $N_L(T)$ of $T$ in $\Delta$ is $T\rtimes U$ with the (non-cyclic, abelian) subgroup $U$ of order $2(q-1)$ generated by $E$ and $\Phi$.

 If $Q_1$ and $Q_2$ are $p$-subgroups of $\aut(\cX)$ fixing at least one of the points $P_1,P_2$, then $Q_1$ and $Q_2$ generate a $p$-subgroup of $\aut(\cX)$. Therefore, there is a unique $p$-subgroup $Q\leq \aut(\cX)$ of maximal order which fixes both $P_1$ and $P_2$.

 \begin{lem}
 \label{lem29nov2018} The Sylow $p$-subgroups of $\aut(\cX)$ are conjugate to the Sylow $p$-subgroups of $H$.
 \end{lem}
 \begin{proof} There is a Sylow $p$-subgroup $S$ of $\aut(\cX)$ which contains $T$. Our aim is to prove that $T=S$. By way a contradiction, assume that $|S|=p^t|T|=p^tq$ with $t\geq 1$.  We show that some subgroup $Q$ of $S$ properly containing $T$ still fixes both points $P_1,P_2$, and that its normalizer $N(Q)$ in $\aut(\cX)$ contains $E$ and $\Phi$.
 For this purpose, two cases are treated separately according as the center $Z(S)$ of $S$ contains some nontrivial elements of $T$ or it does not.

 In the former case, let $z\in Z(S)\cap T$ be a nontrivial element. If $s\in S$ then $zs=sz$ yields $z(s(P_1))=s(z(P_1))=s(P_1)$ which shows that $s(P_1)$ is fixed by $z$. Similarly, $s(P_2)$ is fixed by $z$. Therefore, either $s(P_1)=P_1,\, s(P_2)=P_2$, or $s(P_1)=P_2,\,s(P_2)=P_1$, that is, $s$ interchanges $P_1$ with $P_2$. Since $s$ has odd order, this cannot occur, and hence $s(P_1)=P_1$ and $s(P_2)=P_2$ for every $s\in S$. Therefore, $S$ itself is an admissible choice for $Q$, and we put $Q=S$ for this case. Since $E_\mu$ fixes both $P_1$ and $P_2$, $E_\mu^{-1}SQE_\mu$ is a Sylow $p$-subgroup which fixes $P_1$. Thus $E\mu^{-1}QE=Q$, that is, $E$ is contained in  $N(Q)$. Similarly, since $L_0$ interchanges $P_1$ and $P_2$, we have $\Phi^{-1} Q\Phi$ is a Sylow $p$-subgroup fixing $P_1$ whence $\Phi\in N(Q)$.

If $Z(S)\cap T=\{id\}$ then the subgroup generated by $T$ and $Z(S)$ is their direct product, so that $T\times Z(S)$ is an abelian group. We show that $T\times Z(S)$ is an admissible choice for $Q$. In fact, since $Z(T\times Z(S))=T\times Z(S)\gneqq  T$, the above argument still works when $S$ is replaced by $T\times Z(S)$ showing that $T\times Z(S)$ fixes $P_1$ and $P_2$. Now, define $Q$ to be a $p$-subgroup of $\aut(\cX)$ whose order is maximal with respect to the following two properties:
\begin{itemize}
\item[(i)] $Q$ fixes both $P_1$ and $P_2$.
\item[(ii)] $Q$ contains $T$.
\end{itemize}
Since $T\times Z(S)$ satisfy both (i) and (ii), we have that $|Q|>|T|$. Furthermore, $Q$ is uniquely determined. In fact, since all elements in $\aut(\cX)_{P_1}$ whose order is  a power of $p$ are contained in a unique $p$-subgroup, the $p$-subgroups of $\aut(\cX)$ satisfying (i) and (ii) generate a subgroup that still satisfies (i) and (ii). Let $|Q|=p^rq$ with $r\geq 1$. Furthermore, since $Q$ is uniquely determined, and $E_\mu$ fixes both $P_1$ and $P_2$, we have $E\le N(Q)$. Since $\Phi$ interchanges $P_1$ and $P_2$, we also have $\Phi \in N(Q)$.

To finish the proof, we use an idea from the proof of \cite[Lemma 4.1]{GK2018}. Let $\bar{\cX}=\cX/Q$ be the quotient curve of $\cX$ by $Q$, and denote by $\bar{F}$ its function field. Since $E,L_0\le N(Q)$,  $\aut(\bar{\cX})$ has a subgroup $\bar{U}$ of order $2(q-1)$ isomorphic to $E\times \langle L_0 \rangle$. Clearly, $\bar{U}$ fixes the point of $\bar{\cX}$ lying under $P_1$. Moreover, $\bar{\cX}$ is not rational, as $\bar{U}$ is abelian but non-cyclic. Therefore, $\mathfrak{g}(\bar{\cX})\geq 1$.  Then  $4\mathfrak{g}(\bar{\cX})+2\geq |U|$,  see \cite[Theorem 11.60]{hirschfeld-korchmaros-torres2008}.  Furthermore,
the Riemann-Hurwitz genus formula applied to $Q$ gives $\mathfrak{g}-1\geq |Q|(\mathfrak{g}(\bar{\cX})-1)+2(|Q|-1)$. Therefore,
$$
\begin{array}{lll}
(4\mathfrak{g}(\bar{\cX})+2)q\geq |\bar{U}||Q|=2p^rq(q-1)=2p^r(q^2-q-2)+4p^r=2p^r(\mathfrak{g}-1)+4p^r\geq \\
2p^r(p^rq(\mathfrak{g}(\bar{\cX})-1)+2(p^rq-1))+4p^r>2p^{2r}q\mathfrak{g}(\bar{\cX})-2p^{2r}q-4p^r+2(p^rq-1))+4p^r,
\end{array}
$$
whence $$\mathfrak{g}(\bar{\cX})\leq \frac{(p^{2r}-2)q-(p^rq-1)-2p^r+6q}{(p^{2r}-2)q}<1,$$
a contradiction. \end{proof}
As a corollary, we have the following claim.
\begin{lemma}
\label{lem29nov2018A} Every Sylow $p$-subgroup of $\aut(\cX)$ fixes exactly two points on $\cX$ and they are interchanged by $\Phi$.
\end{lemma}
\begin{lem}
 \label{lem10jun2019} $\aut(\cX)$ has a unique non-tame orbit.
 \end{lem}
 \begin{proof} Take a point $P$ lying in a non-tame orbit of $\aut(\cX)$, and let $S$ be the first ramification group of $\aut(\cX)$ at $P$. Since $S$ is contained in Sylow $p$-subgroup of $\aut(\cX)$, Lemma \ref{lem29nov2018} shows that the conjugate of $S$ by some element $u\in\aut(\cX)$ is contained in the Sylow $p$-subgroup $T$ of $\aut(\cX)$. Therefore, the image of $P$ by $u$ is a fixed point of $T$, and hence it is either $P_1$ or $P_2$. In particular, $P$ and $P_1$ (or $P_2$) are in the same orbit of $\cX$. Since $\Phi$ interchanges $P_1$ and $P_2$, the latter two points are in the same orbit. It turns out that $P$ is in the orbit of $P_1$ whence the assertion follows.
 \end{proof}
 \begin{lem} \label{ramif}
 The second ramification group of $\aut(\cX)_{P_1}$ at $P_1=(1:0:0)$ is trivial.
\end{lem}
 \begin{proof} According to Lemma \ref{lem091021}(iii), $X_\infty$ is a ($q-1$)-fold singular point whose (unique) tangent is the line $\ell_\infty$ at infinity $Z=0$, and $I(P_1,C\cap\ell_\infty)=q$ whereas $I(P_1,C\cap\ell_1)=q-1$ for the line $\ell_1$ of equation $Y=0$. Therefore, $I(X_\infty,\cC_\lambda\cap\ell_\infty)-I(X_\infty,\cC_\lambda \cap\ell_1)=1$. This shows that $y^{-1}$ is a local parameter of $F$ at the point $P_1$.

From Lemma \ref{lem29nov2018}, $\Gamma_{P_1}^{(1)} = T$. For $T_\alpha \in T$ with $\alpha\in \mathbb{F}_q^*$,
$$
\begin{array}{llll}
v_{P_{1}}(T_\alpha(y^{-1}) -y^{-1}) = v_{P_{1}}(({y+\alpha})^{-1}-y^{-1}) = v_{P_{1}}(\alpha y^{-1})+v_{P_{1}}(\alpha(y+\alpha)^{-1})=\\
v_{P_{1}}(ay^{-1})+v_{P_{1}}(ay^{-1}(1+ay^{-1})^{-1})=2v_{P_{1}}(ay^{-1})+v_{P_{1}}((1+ay^{-1})^{-1})=\\
2+v_{P_{1}}(1-ay^{-1}+(ay{-1})^2+\ldots)=2+0=2,
\end{array}
$$
whence the claim follows. \end{proof}
 \begin{lem}
 \label{lemA10jun2019} Let $P$ be a point in the unique non-tame orbit of $\aut(\cX)$. Then the stabilizer of $P$ in $\aut(\cX)$ has order $q(q-1)$ and it is the semidirect product of a Sylow $p$-subgroup by a cyclic group of order $q-1$. In particular, if $P$ is a fixed point of $T$ then the stabilizer of $P$ in $\aut(\cX)$ is contained in $H\times \langle \Phi \rangle$, and the centralizer of $T$ in $\aut(\cX)$ has order $2$.
 \end{lem}
 \begin{proof} We may assume that $P=P_1$. Then the stabilizer of $P$ in $\aut(\cX)$ is $T\rtimes E$.  By way of contradiction, assume that $|V|>|E|=q-1$. Since $V$ is contained in the normalizer $N(T)$ of $T$ in $\aut(\cX)$, $V$ acts by conjugation on the set of nontrivial elements of $T$. Therefore, the assumption $|V|>q-1$ yields that some non-trivial element $v\in V$ commutes with some non-trivial element of $t\in T$. From \cite[Lemma 11.75(i)]{hirschfeld-korchmaros-torres2008},  the second ramification group of $\Gamma_{P_1}$ contains $t$, a contradiction with Lemma \ref{ramif}.
 \end{proof}
\begin{lemma}
\label{lem211021} There is a point $P\in \cX$ such that the intersection of $H$ with the stabilizer of $P$ in $\aut(\cX)$  has order $2$.
\end{lemma}
\begin{proof} According to Propositions \ref{pro111021} and \ref{pro241021}, take a point $P\in \cX$ such that the point $\bar{P}\in\cD_\lambda$ lying under $P$ is in a tame orbit of $\aut(\cD_\lambda)=\PGL(2,q)$. Then the stabilizer of $\bar{P}$ is an involution. Since $\Phi$ centralizes $H$, there is an involution in $H\times \langle \Phi \rangle$ which fixes a point of $\cX$. Since this fixed point is off the unique non-tame orbit of $\aut(\cX)$, it is a point of a tame orbit of $\aut(\cX)$.
\end{proof}
\begin{theorem}
 \label{prop10jun2019} Let $\cX$ be a nonsingular model of the plane curve $\cC_\lambda$ with homogeneous equation (\ref{eq00}). If  $q>3$ and $\lambda \not\in \{-1,1\}$ then
 $
 \aut(\cX) = H \times \langle \Phi \rangle.
  $
 In particular, $\aut(\cX)\cong \PGL(2,q)\times C_2$, and it has order $2(q^3-q)$.
 \end{theorem}
 \begin{proof} According to Lemma \ref{lem10jun2019}, let $\Omega$ be the unique non-tame orbit of $\aut(\cX)$. Assume on the contrary $|\aut(\cX)>|H||\langle \Phi \rangle|=2(q^3-q)$. Then $|\aut(\cX)|=2d(q^3-q)$ for an integer $d\ge 2$ so that $d=[\aut(\cX):H\times \langle \Phi \rangle]$. Actually $d\ge 3$, otherwise $H\times \langle \Phi \rangle$ is a normal subgroup of $\aut(\cX)$, and hence $\aut(\cX)$ preserves the unique non-tame orbit of $H\times \langle \Phi\rangle$. But then in that orbit the $1$-point stabilizer of $\aut(\cX)$ would be larger than the $1$-point stabilizer of $H\times \langle \Phi\rangle$, contradicting Lemma \ref{lemA10jun2019}.
 This argument also shows $|\Omega|=2dq(q-1)$. Furthermore, $d\neq p$ by Lemma \ref{lem29nov2018}.

 Let $P$ denote either $P_1$  or $P_2$. Then $P\in \Omega$. Since $T$ is the Sylow $p$-subgroup of $\Gamma$ which fixes $P$ and $E$ is its complement in the stabilizer of $P$, Lemmas \ref{lemA10jun2019} and \ref{ramif} yield that the degree of the different $d_P$ is equal to $|E|(|T|-1)+|T|-1=|E||T|+|T|-2 = q^2-2$. Moreover, Lemma \ref{lem211021} shows that $\aut(\cX)$ has a tame short orbit. If $\aut(\cX)$ has $k$ tame short orbits then the Riemann-Hurwitz formula applied to the cover $\cX\rightarrow \cX/\aut(\cX)$ reads
$$
 2q^2-2q-4 = -4d(q^3-q)+ 2d(q+1)(q^2-2) + c_1(e_{Q_1}-1)+\ldots c_k(e_{Q_k}-1)
$$
where $e_{Q_i}$ stands for the ramification index of any representative point $Q_i$ from a $\Gamma$-orbit of length $c_i$, and $c_i\ge 1$ is the length of the $\Gamma$-orbit containing $Q_i$. Therefore, $c_1(e_{Q_1}-1)+\ldots c_k(e_{Q_k}-1)>0$.
From the proof of Cases (I) and (II) in \cite[Theorem 11.56]{hirschfeld-korchmaros-torres2008}, as $|\aut(\cX)|\ge |H\times \langle \Phi \rangle|=2(q^3-q)>12(q^2-q-2)=12(\mathfrak{g}-1)$, we have $k\le 2$.


We are left with two cases, namely $k=1$ and $k=2$ respectively.

If $k=1$, then the Riemann-Hurwitz formula applied to $\aut(\cX)$ yields
\begin{equation}
\label{eq24oct21}
2q^2-2q-4 = -4d(q^3-q) +2d(q+1)(q^2-2)  + 2d(q^3-q)(d^*-1)/d^*
\end{equation}
where $d^*$ denotes the order of the stabilizer of a point in the
unique tame orbit. We show that $d^*$ divides $2d$. Let $\Delta_1$ and $\Delta_2$ denote the orbits of $P$ in $\aut(\cX)$ and in $H\times \langle \Phi \rangle$, respectively. Then $|\Delta_2|$ divides $|\Delta_1|$ as $H\times \langle \Phi \rangle$ is a subgroup of $\aut(\cX)$, and the claim follows from $d^*|\Delta_1|=2d(q^3-q)$ and $2|\Delta_2|=2(q^3-q)$. Now, (\ref{eq24oct21}) can be written as
$$2q^2-2q-4 = -4d(q^3-q) +2d(q+1)(q^2-2)  + \frac{2d}{d^*}(q^3-q)(d^*-1)$$ whence
  \begin{equation}\label{2orbits}
 (q-2)(d-1) = \frac{2d}{d^*}\frac{q(q-1)}{2}.
 \end{equation}
 Since both $2d/d^*$ and $\ha q(q-1)$ are integers, ${\rm{g.c.d.}}(2d/d^*,d-1)\le 2$ and ${\rm{g.c.d.}}(q-2, \ha q(q-1))=1$ for $q>3$, Equation  (\ref{2orbits}) implies $d/d^*=q-2$. Thus, (\ref{2orbits}) reads $d-1=q(q-1)$ whence $d=q^2-q+1$. Therefore, $q-2$ divides $q^2-q+1$ which is impossible for $q>3$.

 Finally, let $k =2$. From the proof of \cite[Theorem 11.56]{{hirschfeld-korchmaros-torres2008}}, either $e_{Q_1}=e_{Q_2}=2$, or $|\aut(\cX)|<84(\mathfrak{g}(\cX)-1)$.
 In the former case, the Riemann-Hurwitz formula applied to $\aut(\cX)$ yields
$$
 2q^2-2q-4 = -4d(q^3-q) +2d(q+1)(q^2-2) + 2d(q^3-q) = 2d(q^2-q-1),
$$
 whence $d = 1$ follows, a contradiction. Therefore $2d(q^3-q)=|\aut(\cX)|\le 84(\mathfrak{g}(\cX)-1)<84(q^2-q)$ whence $2d(q+1)<84$ follows. By $d\ge 3, d\neq p$, this implies $3\le d \le 6$. For $i=1,2$, let $P_i$ be a point in the short tame orbit $\Delta_i$ of $\aut(\cX)$, and let $d_i^*$ be the order of the stabilizer of $P$ in $\aut(\cX)$. As we have shown in the proof of case $k=1$, this yields $d_i^*|2d$. From the Riemann-Hurwitz
 formula,
 $$
 2q^2-2q-4 = -4d(q^3-q) +2d(q+1)(q^2-2) + \frac{2d}{d_1^*}(q^3-q)(2d_1^*-1)+\frac{2d}{d_2^*}(q^3-q)(2d_2^*-1),
$$
whence $d\equiv 1 \pmod q$ follows. Since $3\le d \le 6$ and $q>3$ is odd, it follows that either $d=4$ and $q=3^h$, or $d=6$ and $q=5^h$. This together with $2d(q+1)<84$ and $d\neq p$ imply $d=4$ and $q=9$. Finally, a direct computation rules out this possibility.
\end{proof}

\begin{proposition}
\label{pro251021} Let $T$ be a Sylow $p$-subgroup of $H$. Then the quotient curve $\cX/T$ is isomorphic to the genus $q-2$ hyperelliptic curve $\cF_\lambda$  of affine equation
\begin{equation}
\label{eq281121} W^2=V^{2q-2}+2\lambda V^{q-1}+1.
\end{equation}
If $\gamma(\cX)$ and $\gamma(\cX/T)$ are the $p$-ranks of $\cX$ and $\cX/T$ respectively, then $\gamma(\cX)=q(\gamma(\cX/T))+q-1.$ In particular, $\cX$ is ordinary if and only if $\cX/T$ is ordinary.
\end{proposition}
\begin{proof} As we have already pointed out, the group $T=\{T_\alpha|\alpha\in \mathbb{F}_q\}$ has two fixed points on $\cX$, namely $P_1$ and $P_2$. Proposition \ref{genusBKSc} together with Lemma \ref{ramif} and the Riemann-Hurwitz formula applied to $T$ give that $\cX/T$ has genus $q-2$.
Let $F_0=F(u,v)$ be the subfield of $F$ which is generated by $u=x^q-x$ and $v=x-y$. Then $[F:F_0]=q$ and $F_0$ is the fixed field of $T$. Furthermore,
$$
\begin{array}{lll}
(x+y)^{q+1}-2((xy)^{q}+xy)-\lambda (x-y)^{q+1}=(2x-v)^{q+1}-2(x^q(x-v)^q+x(x-v)-\lambda v^{q+1}=\\
2u^2-2(v^q-v)u-(\lambda-1)v^{q+1}=2(u-\ha(v^q-v))^2-\frac{1}{2}(v^q-v)^2-(\lambda-1)v^{q+1}=\\
2(u-\ha(v^q-v))^2-\frac{1}{2}(v^{2q}+v^2)-\lambda v^{q+1}=\frac{1}{2}[(2(u-\ha(v^q-v))v^{-1})^2-v^{2q-2}-2\lambda v^{q-1}-1].\\
\end{array}
$$
whence, with $w=(2(u-\ha(v^q-v))v^{-1})$, the first claim follows. Now, from the Deuring-Shafarevic formula \cite[Theorem 11.62]{hirschfeld-korchmaros-torres2008}, $\gamma(\cX)-1=q(\gamma(\cX/T)-1)+2(q-1)$, whence the second claim follows. Comparison with
$\mathfrak{g}(\cX)-1=q(\mathfrak{g}(\cX/T)-1)+2(q-1)$ shows the last claim.
\end{proof}
\begin{rem} \em{As $\cF_\lambda$ has genus $\bar{\mathfrak{g}}=q-2$, its equation (\ref{eq281121}) is $W^2=V^{2\bar{\mathfrak{g}}+2}+\tau V^{\bar{\mathfrak{g}}+1}+1$ with $\tau=2\lambda$. It should be noticed that the already quoted paper of Miller \cite{Mi} also concerned the Hasse-Witt matrix of hyperelliptic function fields of equation $w^2=v^{2\bar{\mathfrak{g}}+2}+\tau v^{\bar{\mathfrak{g}}+1}+1$. Unfortunately, his claim stating invertibility apart from finitely many values of $\tau$, does not apply to $\cF_\lambda$  since that claim, \cite[Proposition 2]{Mi}, was proven under the extra-condition that $p\mid \bar{\mathfrak{g}}$ whereas $\bar{\mathfrak{g}}=q-2$ in our case. Thus, it remains open the problem of finding values of $\lambda$ for which the curves $\cX_\lambda$ is ordinary.}
\end{rem}

  We just need to consider the case $q = 3$; in this case, we prove the following.

  \begin{theorem}
 Theorem \ref{prop10jun2019} holds true for $q=3$, with a unique exception, $\lambda=0$, in which case $\aut(\cX)\cong SmallGroup(192,181)$.
  \end{theorem}
\begin{proof} We use some ideas and results from \cite{homma1980} and \cite{nazarpietro}. From Proposition \ref{genusBKSc}, $\cX$ is a genus $5$ curve. Assume that $\aut(\cX)$ is larger than $H\times \langle \Phi \rangle$. As we have pointed out at the beginning of the proof of Theorem
\ref{prop10jun2019}, $H\times \langle \Phi \rangle$ is not a normal subgroup of $\aut(\cX)$. Therefore, $|\aut(\cX)|=d|H\times \langle \Phi \rangle| =48d$ with $d\ge 4$. In particular, $|\aut(\cX)|>48=12(\mathfrak{g}(\cX)-1)$ and hence $\aut(\cX)$ has at most two tame short orbits.
Let $u>3$ be a prime dividing $|\aut(\cX)|$, and take a subgroup $U$ of $\aut(\cX)$ of order $u$. Since $2\mathfrak{g}(\cX)-2=8$ and $\rm{g.c.d.}(u,8)=1$, the Riemann Hurwitz formula applied to $U$ shows that $U$ has a fixed point on $\cX$. From Lemma \ref{lemA10jun2019}, any fixed point $P$  of $U$ is in a tame short orbit $\mathcal{O}$ of $\aut(\cX)$. By Lemma \ref{lem211021}, some point in a tame short orbit $\mathcal{Q}$ is fixed by an involution $J\in H\times \langle \Phi \rangle$. By the Riemann-Hurwitz formula, $J$ has either $4$, or $8$, or $12$ fixed points.  The latter case can be discarded since happens as it would imply that $\cX$ is hyperelliptic, which is not the case. If $\mathcal{O}=\mathcal{Q}$ then the stabilizer of $P$ has a cyclic subgroup of order $2u$ containing $U$. Therefore, $J$ centralizes $U$. Otherwise, $\mathcal{O}$ and $\mathcal{Q}$ are the two tame short orbits of $\cX$, and $|\aut(\cX)|\le 84(\mathfrak{g}(\cX)-1)=366$ by the proof of case (III) of \cite[Theorem 11.56]{hirschfeld-korchmaros-torres2008}. In the latter case, $d\le 7$.

From \cite[Theorem 1, Theorem 2(c)]{homma1980}, either $u=5$, or $u=11$. We investigate these possibilities separately.

If $u=11$, then \cite[Theorem 10]{nazarpietro} yields that $|\aut(\cX)| = 11\,c$, for $c \in \{1,2,3\}$ but this contradicts that $|\aut(\cX)| \geq 48$.

If $u=5$, then \cite[Proposition 11]{nazarpietro} shows that $U$ fixes two points on $\cX$ and that $\cX/U$ is elliptic.
Assume that $\aut(\cX)$ has a unique tame short orbit.
Since at least one fixed point of $U$ is also fixed by $J$, it turns out that $J$ has as many as $2+5m$ fixed points on $\cX$ where $m\ge 0$. But then $J$ has either $2$, or $7$, or $12$ fixed points, a contradiction.
Therefore, $\aut(\cX)$ has two short orbits, say  $\mathcal{O}$ and $\mathcal{Q}$. Then $|\aut(\cX)|\le 84(\mathfrak{g}(\cX)-1)=366$ implies $|\aut(\cX)|=240$, and the Riemann-Hurwitz formula applied to $\aut(\cX)$ reads $8 = -480+ 280 + (240-|\mathcal{O}|)+ (240-|\mathcal{Q}|),$
whence $|\mathcal{O}|+|\mathcal{Q}| =272$. But this contradicts  $|\mathcal{O}|+ |\mathcal{Q}|\leq 120 +48$.

We may assume $|\aut(\cX)|=192$. Take a point $P\in\cX$ fixed by the subgroup $T$ of $H$ order $3$. From Lemma \ref{lemA10jun2019}, the normalizer of $T$ has order $12$ and its center has order $2$. Therefore,
\begin{enumerate}
\item[(i)] the normalizer of a Sylow $3$-subgroup of $\aut(\cX)$ is a dihedral group of order $12$.
\end{enumerate}
Moreover, the subgroup $H\times \langle \Phi \rangle$ of $\aut(\cX)$ is not normal. In fact, otherwise, all Sylow $3$-subgroups of $\aut(\cX)$ would be contained in $H$, and hence their number would be four. On the other hand, from (i), that number is equal to $192/12=16$.
 \begin{enumerate}
\item[(ii)] $\aut(\cX)$ contains a non-normal subgroup isomorphic to $\PGL(2,3)\times C_2$.
\end{enumerate}
Let $S_2$ be a Sylow $2$-subgroup of $\aut(\cX)$. Then $S_2$ has order $64$. From \cite[Remark 3.5]{anbar}, $S_2$ has three orbits on $\cX$ and they have lengths $32,16$ and $8$ respectively. Therefore, if $P$ is a point in the $S_2$-orbit of length $8$, then the stabilizer of $P$ in $S_2$ is a cyclic group of order $8$. Therefore,
\begin{enumerate}
\item[(iii)] The exponent of $\aut(\cX)$ is at least $8$.
\end{enumerate}
By an exhaustive search supported by the MAGMA algebra system, there exist, up to isomorphisms, five groups of order $192$ satisfying the above three properties, each but one has center of order $2$. The exception is named $SmallGroup(192,956)$, and it has trivial center. We show that the exceptional case cannot actually occur in our case. Assume on the contrary $\aut(\cX)\cong SmallGroup(192,956)$. The MAGMA database shows that $\aut(\cX)$ has a unique normal subgroup $N$ of order $48$, namely $SmallGroup(48,3)$. Thus,   the Sylow $3$-subgroups of $N$ are self-normalizing in $N$. Therefore, $N$ has as many as $16$ Sylow $3$-subgroups. Also, each of them has exactly two fixed points by Lemma \ref{lem29nov2018A}.
From the Riemann Hurwitz formula applied to $N$, it follows $8=2\mathfrak{g}(\cX)-2\ge 96 (\mathfrak{g}(\cX/N)-1)+32\cdot 4$,
a contradiction. Therefore, the center of $\aut(\cX)$ contains a unique involution $\iota$. Actually $\Phi=\iota$ holds, since the last claim in Lemma \ref{lemA10jun2019} shows that $\Phi$ is the unique involution in the centralizer of the particular Sylow $3$-subgroup $T$ of $\aut(\cX)$.

Finally, from Proposition \ref{pro111021}, $\aut(\cX)/\langle\Phi \rangle$ is a subgroup of $\aut(\cD_\lambda)$ of order $96$. On the other hand, Theorem \ref{th131021} shows that $|\aut(\cD_\lambda)|=|\PGL(2,3)|=48$, with a unique exception for $\lambda=0$. In the exceptional case, $\aut(\cX)\cong SmallGroup(192,181)$, and hence the center of $\aut(\cX)$ has an elementary abelian subgroup $C_2\times C_2$, so that $\aut(\cX)$ is a central (non-split) extension of $\GL(2,3)$ by $C_2\times C_2$.
\end{proof}
\begin{rem}
\label{rem2nov21}{\em{We point out a connection between $\cC_0$, i.e. the irreducible plane sextic of affine equation $(X+Y)^4-2((XY)^3+XY)=0,$ and the well known genus $5$ complex curve $W_{192}$ of Wiman  of
affine equation $Y^4-X^2(X^4-1)=0$, see \cite{Wi,hk,kato}. The notoriety of $W_{192}$ stems from its property of attaining the maximum number of automorphisms that a genus $5$ curve can possess.
Actually, in characteristic $3$, the above equation still defines a nonsingular curve of genus $5$ whose automorphism group is isomorphic to that of $W_{192}$. A MAGMA computation shows that this curve is isomorphic to $\cC_0$ over $\mathbb{F}_{81}$. It should be stressed however that the classical approach used in \cite{Wi}, as well as its variants  and other approaches also depending on tools from modern algebraic geometry \cite{bcg,centina,hk,kato,kato1,km}, do not seem plausible to extend straightforwardly to our case mostly due to the fact that $\aut(\cC_0)$ is non-tame in characteristic $3$, and hence the different in the Hurwitz Riemann formula changes passing from characteristic zero to characteristic $3$. It would be interesting to study the existence of some more genus $5$ curves in characteristic $3$ with $192$ automorphisms.
}}
\end{rem}


\begin{thebibliography}{999}
\bibitem{anbar} N. Anbar and B. G\"unes, On nilpotent automorphism groups of function fields, to appear in {\emph{Adv. Geom.}}, arXiv:1912.08146 [math.AG], 2019.

\bibitem{anazar} N. Arakelian, Separable degree of the Gauss map and strict dual curves over finite fields,
\emph{Bull. Braz. Math. Soc. (N.S.)} {\bf{52}} (2021),  135-148.

\bibitem{nazarpietro} N. Arakelian and P. Speziali, Algebraic curves with automorphism groups of large prime order, \emph{Math. Zeit.} {\bf 299} 2021, 2005-2028.


\bibitem{bloom} D. M. Bloom. The Subgroups of $\PSL(3,q)$ for odd $q$, \emph{Trans. Amer. Math. Soc.}, {\bf 127} (N. 1) (1967) 150-178.

\bibitem{bor} H. Borges, On multi-Frobenius nonclassical plane curves, \emph{Arch. Math. (Basel)} {\bf{93}} (2009),  541-553.

\bibitem{bor-hom} H. Borges and M. Homma, Points on singular Frobenius nonclassical curves, \emph{ Bull. Brazil. Math. Soc.}, {\bf 48} (2017), 93--101.

\bibitem{MAGMA} W. Bosma, J. Cannon, and C. Playoust, The Magma algebra system. I. The user language, \emph{J. Symbolic Comput.}, {\bf 24} (1997), 235-265.

\bibitem{bcg} E. Bujalance, F.J. Cirre, and G. Gromadzki, G. Groups of automorphisms of cyclic trigonal Riemann surfaces.
\emph{J. Algebra} {\bf{322}} (2009), 1086-1103.

\bibitem{centina} A. Del Centina, On certain remarkable curves of genus five, \emph{Indag. Mathem., N.S.,} {\bf{15}} (2004), 339-346.

\bibitem{CS1} A. Cossidente and A. Siciliano, Plane algebraic curves with Singer automorphisms,

\bibitem{CS2} A. Cossidente, A. Siciliano, The automorphism group of plane algebraic curves with Singer automorphisms,



\bibitem{dickson} L.E. Dickson, Linear groups, with an exposition of the Galois field theory, Teubner, Leipzig, 1901.

\bibitem{GK2018} M. Giulietti, G. Korchm\'aros, Algebraic curves with many automorphisms, \emph{Adv. Math.} {\bf{349}} (2019), 162-211.

\bibitem{DGZ} M. Giulietti, G. Korchm\'aros, M. Timpanella, On the Dickson-Guralnick-Zieve curve, \emph{J. Number Theory} {\bf 196} (2019), 114-138.

\bibitem{hartley} R.W. Hartley, Determination of the ternary collineation groups whose coefficients lie in the $GF(2^n)$, \emph{Ann. of Math.} {\bf 27} (1925), 140-158.


\bibitem{hefez-voloch} A. Hefez and J.F.Voloch, Frobenius non classical curves, \emph{Arch. Math.} {\bf 54}, (1990) 263-273.

\bibitem{homma1980} M. Homma, Automorphisms of prime order of curves, \emph{Manuscripta Math.} {\bf{33}} (1980/81),  99-109.

\bibitem{hr} R. Horiuchi, Non-hyperelliptic Riemann surfaces of genus five all of whose Weierstrass points have maximal weight,
{\emph{Kodai Math. J.}} {\bf{30}} (2007),  379-384.
\bibitem{hk} R. Horiuchi, Ch. Keem, Changho  Defining equations of curves of genus five possessing two Weierstrass points of maximal weight with a common $g_4^1$,
\emph{Georgian Math. J.} {\bf{18}} (2011),  705-725.

\bibitem{kato} T. Kato, K. Magaard and H. V\"olklein, Bi-elliptic Weierstrass points on curves of genus $5$, \emph{Indag Math.} {\bf{22}} (2011), 116-130.

\bibitem{kato1} T. Kato, A curve of genus $5$ having $24$ Weierstrass points of weight $5$, \emph{Hokkaido Math. J.} {\bf{44}} (2015),  165-173.
\bibitem{hirschfeld-korchmaros-torres2008} J.W.P. Hirschfeld, G. Korchm\'aros and F. Torres,\emph{  Algebraic Curves Over a Finite Field}, Princeton Univ. Press, Princeton, (2008), xiii + 720 pp.

\bibitem{km} Ch. Keem, G. Martens, On curves with all Weierstrass points of maximal weight, \emph{Arch. Math. (Basel)} {\bf{94}} (2010),  339-349.

\bibitem{king} O.H.~King, The subgroup structure of finite classical groups in terms of geometric configurations, in \emph{Surveys in combinatorics} 2005, 29-56, London Math. Soc. Lecture Note Ser., {\bf{327}}, Cambridge Univ. Press, Cambridge, 2005.

\bibitem{KM}G. Korchm\'aros, M. Montanucci, Ordinary algebraic curves with many automorphisms in positive characteristic, \emph{Algebra Number Theory} {\bf 13} (2019), 1-18.

\bibitem{KNS2017} G. Korchm\'aros, G. Nagy, P. Speziali, Hemisystems  of the Hermitian Surface in $PG(3,q^2)$,  \emph{ J. Combin. Theory Ser. A} {\bf 165} (2019) 408-439.


\bibitem{KK} A. Kuribayashi, H.  Kimura, Automorphism groups of compact Riemann surfaces of genus five.
\emph{J. Algebra} {\bf{134}} (1990),  80-103.

\bibitem{Mi} L. Miller, Curves with Invertible Hasse-Witt-Matrix, \emph{Math. Ann} {\bf{197}} (1972), 123-127.

\bibitem{mitchell} H. Mitchell, Determination of the ordinary and modular linear form, \emph{Trans. Amer. Math. Soc.} {\bf 12} (1911) 207-242.

  \bibitem{MS1}  M. Montanucci,  P. Speziali, Large automorphism groups of ordinary curves in characteristic $2$,  \emph{J. Algebra},  {\bf 526} (2019) 30-50.



\bibitem{Pel} R. Pellikaan, The Klein quartic, the Fano plane and curves representing designs, in \emph{Codes, curves, and signals} (Urbana, IL, 1997), 9,
Kluwer Internat. Ser. Engrg. Comput. Sci., 485, Kluwer Acad. Publ., Boston, MA, 1998.

  \bibitem{stherm} H. Stichtenoth, \"Uber die Automorphismengruppe eines algebraischen F\"unktionenk\"orpers von Primzahlcharacteristic. I. Eine Absch\"atzungder Ordnung der AutomorphismenGruppe, \emph{Arch. Math.}
  {\bf 24} (1973), 527-544.


H.~Stichtenoth,  \emph{Algebraic function fields and codes}, Springer-Verlag, Berlin and Heidelberg, (1993), vii+260 pp.

\bibitem{SV} K.O. St\"ohr and J.F. Voloch, Weierstrass points on curves over finite fields, \emph{Proc. Lond. Math. Soc.} {\bf{52}} (1986), 1-19.

\bibitem{Wi} A. Wiman, \"Uber die algebraischen Kurven von Geschlechtern 4, 5 und 6, welche eindeuitige Transformationen in sich besitzen, Bihang till Kongl. Svevska Vetenskaps-Akademiens Handlingar, {\bf{21}} Nr. 3, (1895).
\end{thebibliography}
\end{document}